\documentclass[oneside,a4paper,12pt,notitlepage]{article}
\usepackage[left=0.9in, right=0.9in, top=1.5in, bottom=1.5in]{geometry}
\usepackage[T1]{fontenc} 
\usepackage[utf8]{inputenc} 
\usepackage[english]{babel} 
\usepackage{lipsum} 
\usepackage{lmodern}
\usepackage{amssymb}
\usepackage{amsthm}
\usepackage{bm}
\usepackage{mathtools}
\usepackage{xcolor}
\usepackage{dsfont}
\usepackage{enumerate}

\usepackage{yhmath}
\usepackage{stmaryrd}
\usepackage{amsmath}
\usepackage{tikz}
\usepackage{verbatim}
\usepackage{graphicx}
\usepackage{enumitem}
\usepackage{footnote} %

\usepackage{braket}
\usepackage{esint}
\newcommand{\abs}[1]{{\left|#1\right|}}
\newcommand{\norma}[1]{{\left\Vert#1\right\Vert}}

\usepackage{booktabs}
\usepackage{graphicx}
\usepackage{tikz}
\usetikzlibrary{patterns}
\usepackage{multicol}
\usepackage{caption}
\usepackage{enumerate}
\usepackage[skins,theorems]{tcolorbox}
\tcbset{highlight math style={enhanced,
		colframe=black,colback=white,arc=0pt,boxrule=1pt}}
\theoremstyle{definition}
\newtheorem{definizione}{Definition}[section]
\theoremstyle{plain}
\newtheorem{teorema}{Theorem}[section]

\newtheorem{lemma}[teorema]{Lemma}
\newtheorem{prop}[teorema]{Proposition}
\newtheorem{corollario}[teorema]{Corollary}
\theoremstyle{definition}
\newtheorem{esempio}{Example}[section]
\newtheorem{oss}[esempio]{Remark}

\DeclareMathOperator{\R}{\mathbb{R}}

\DeclareMathOperator*{\argmin}{\text{arg}\,\text{min}}
\usepackage{hyperref}
\hypersetup{linktoc=none, bookmarksnumbered, colorlinks=true, linkcolor=magenta, citecolor=cyan}

\title{On the second eigenvalue of the infinity Laplacian with Robin boundary conditions}
\author{Vincenzo Amato, Alba Lia Masiello, Carlo Nitsch, Cristina Trombetti}
\date{\today}

\newcommand{\Addresses}{{
 \bigskip 
 \footnotesize 
 \noindent \textit{E-mail address}, V.~ Amato: \texttt{v.amato@ssmeridionale.it} 
  
   \medskip 
 
  \noindent\textsc{Mathematical and Physical Sciences for Advanced Materials and Technologies, Scuola Superiore Meridionale, Largo San Marcellino 10, 80138 Napoli, Italy. }

 \medskip

 \textit{E-mail address}, A.L.~Masiello: \texttt{masiello@altamatematica.it}

\noindent \textsc{
 	Holder of a research grant from Istituto Nazionale di Alta Matematica "Francesco Severi" at Dipartimento di Matematica e Applicazioni "R. Caccioppoli", Via Cintia, Complesso Universitario Monte S. Angelo, 80126 Napoli, Italy.}

 \medskip 
 
 \textit{E-mail address}, C.~Nitsch: \texttt{c.nitsch@unina.it}

  \medskip 
 
 \textit{E-mail address}, C.~Trombetti: \texttt{cristina@unina.it} 
   \medskip

 \noindent\textsc{Dipartimento di Matematica e Applicazioni ``R. Caccioppoli'', Universit\`a degli studi di Napoli Federico II, Via Cintia, Complesso Universitario Monte S. Angelo, 80126 Napoli, Italy.}\par\nopagebreak 

}}

\begin{document}
\maketitle 
	 \vspace{-0.8cm}
	\begin{abstract}
		 We study the behaviour, as $p \to +\infty$, of the second eigenvalues of the $p$-Laplacian with Robin boundary conditions and the limit of the associated eigenfunctions. We prove that, {up to some regularity of the set}, 
   the limit of the second eigenvalues is actually the second eigenvalue of the so-called $\infty$-Laplacian.
		 
	\end{abstract}

\section{Introduction}
Let $\beta$ be a positive parameter and let $\Omega$ be a bounded and open set of $\R^n$, $n\ge 2$, with Lipschitz boundary. 

We aim to study the limit, as $p$ goes to $+\infty$, of $\lambda_{2,p}(\Omega)$, the second eigenvalue of the $p$-Laplace operator with Robin boundary conditions, that is to say, the second smallest value $\lambda$ for which the following problem admits a nontrivial solution \begin{equation*}
\begin{cases}
-\Delta_p u= \lambda \abs{u}^{p-2} u & \text{ in } \Omega \\
\abs{\nabla u}^{p-2} \displaystyle{\frac{\partial u}{\partial \nu}} + \beta^p \abs{u}^{p-2}u =0 & \text{ on } \partial \Omega.
\end{cases}
\end{equation*}

We recall that, for $p \ne 2$, we do not have a complete characterization of all the possible eigenvalues of the $p$-Laplace operator, however, the first two can be precisely defined using a min-max principle. Furthermore, an increasing and divergent sequence of eigenvalues can be constructed, but it is not known whether it contains all the $p$-eigenvalues (see for example \cite{spectrump, Peral2, Peral}).

In a previous paper, \cite{AMNT}, we studied the limit of $\lambda_{1,p}(\Omega)$, the first eigenvalue of the $p$-Laplace operator with Robin boundary conditions. In particular, we proved that
\begin{equation}
    \label{var_carlambda1}
    \lim_{p\to \infty} (\lambda_{1,p}(\Omega))^{\frac{1}{p}}=\lambda_{1,\infty}(\Omega):=\inf_{\substack{w\in W^{1,\infty}(\Omega) \\ \norma{w}_{L^\infty(\Omega)}=1}}
  \max \left\lbrace \norma{\nabla w}_{L^\infty(\Omega)}, \beta\norma{w}_{L^\infty(\partial\Omega)} \right\rbrace,
\end{equation}
and
that the following geometric characterization holds
\begin{equation}
    \label{lambda1}
\lambda_{1,\infty}(\Omega)=\frac{1}{\displaystyle{\frac{1}{\beta} + r(\Omega)}},
\end{equation}
where $r(\Omega)$ is the inradius of the set $\Omega$, i.e., the radius of the largest ball contained in $\Omega$.

Moreover, for a set of class $C^2$, we showed that the right-hand side of \eqref{lambda1} is, in fact, the first eigenvalue of an $\infty$-Laplacian problem, in the sense that it is the smallest value $\lambda$ for which the following problem admits a nontrivial viscosity solution
\begin{equation}
\label{ciao2}
  \begin{cases}
  \min\Set{\abs{\nabla u} - \lambda u, - \Delta_\infty u}=0 &\text{ in } \Omega,\\
  -\min\Set{\abs{\nabla u} - \beta u, -\displaystyle{\frac{\partial u}{\partial \nu}} }=0 &\text{ on } \partial \Omega,
\end{cases}
\end{equation}
where $\Delta_\infty$, the so-called $\infty$-Laplacian, is defined by
$$\Delta_\infty u = \left\langle D^2 u \cdot \nabla u, \nabla u \right\rangle.$$

Hence, our goal is to conduct a similar analysis for $\lambda_{2,\infty}(\Omega)$, obtaining a variational characterization as in \eqref{var_carlambda1} for the limit of the second eigenvalue,  giving it a geometrical characterization as in \eqref{lambda1}, and showing that this limit appears in a PDE problem analogous to \eqref{ciao2}. 

A characterization of $\lambda_{2,\infty}(\Omega)$ can be given in terms of the following geometric quantity
\begin{equation}\label{ess}
   s(\Omega)={\max} \left\{t: \exists x_1, x_2\in \overline{\Omega} \text{ such that }  \abs{x_1-x_2}\ge 2t, \norma{C_{i,t}}_{L^\infty(\partial\Omega)}\le\frac{1}{\beta t}, \, i=1,2  \right\},
\end{equation} 
where $C_{i,t}$ is the function
$$C_{i,t}(x)=\frac{(t-\abs{x-x_i})_+}{t}.$$

We stress that $\norma{C_{i,t}}_\infty=1$, and the function is supported in $\overline{\Omega} \cap B(x_i,t)$, where $B(x_i,t)$ is the ball centered at $x_i$ with radius $t$.

Hence, we prove the following geometric characterization.
\begin{teorema}\label{autolimdiri}
    Let $\Omega$ be a bounded {Lipschitz} set, let $\{\lambda_{2,p}(\Omega)\}$ be the sequence of the second $p$-Laplace eigenvalues, and let $s(\Omega)$ be the quantity defined in \eqref{ess}. Then, we have 

    \begin{equation*}
    \lim_{p\to\infty} (\lambda_{2,p}(\Omega))^{\frac{1}{p}}= \frac{1}{s(\Omega)}=:\lambda^{\beta}_{2,\infty}(\Omega).
    \end{equation*}
\end{teorema}
Throughout the article, when it is clear, we will refer to this quantity only as $\lambda_{2,\infty}(\Omega)$, omitting the superscript $\beta$.

Later, we define  $$\mathcal{S}=\left\{v\in W^{1,\infty}(\Omega): \norma{v}_{L^\infty(\Omega)}=1\right\},$$ and  we consider $u_{1,\infty}$ to be the limit of the first Robin $p$-eigenfunctions, normalized such that $\norma{u_{1,\infty}}_{L^\infty(\Omega)}=1$. 

We have the following variational characterization for $\lambda_{2,\infty}(\Omega)$.
\begin{teorema}
\label{var:carteorema}
Let $\Gamma$ be the set of all the continuous paths in $\mathcal{S}$ joining $u_{1,\infty}$ to $-u_{1,\infty}$. Hence

    \begin{equation*}
 \lambda_{2,\infty}(\Omega)=\displaystyle{\inf_{\gamma\in \Gamma}\sup_{u\in\gamma}\max\left\{\norma{\nabla u}_{L^\infty(\Omega)}, \beta \norma{u}_{L^\infty(\partial\Omega)}\right\}}.
    \end{equation*}
\end{teorema}

Finally, for a particular class of domains, we show that $\lambda_{2,\infty}(\Omega)$ is the second eigenvalue of the $\infty$-Laplacian. To better understand the meaning of "eigenvalues of the $\infty$-Laplacian", we refer to Section \ref{section_notion} and problem \eqref{prob:visc}, which has to be understood in a viscosity sense. 
We then prove the following theorem.
\begin{teorema}\label{secondodavvero}
    Let $\Omega$ be a bounded and convex set of class $C^2$, and let $u$ be a viscosity solution to \eqref{prob:visc} for some $\lambda>0$. Assume that $u$ has at least two nodal domains, then
    $$\lambda\ge \lambda_{2,\infty}(\Omega).$$
\end{teorema}

Unfortunately, for a general Lipschitz set, it is still unclear whether $\lambda_{2,\infty}$ is exactly the second eigenvalue or not.

\vspace{5mm}

This paper continues a series of papers on this topic.
Similar results were obtained in the case of Dirichlet, Robin, and Neumann boundary conditions in \cite{BDM,EKNT,JL,JLM, KN,RS2} and for the Steklov eigenvalue in \cite{gloria, stek}.

To be more precise, in \cite{JL,JLM}, Juutinen, Lindqvist, and Manfredi studied the Dirichlet case as $p \rightarrow +\infty$. They provided a complete characterization of the limiting problems in terms of geometric quantities. 
Indeed, the first two eigenvalues of the $p$-Laplace operator, $\lbrace\lambda_{1,p}^D\rbrace$ and $\lbrace\lambda_{2,p}^D\rbrace$, satisfy
$$
\lim_{p\to \infty} \left(\lambda_{1,p}^D\right)^{{1}/{p}} = \lambda_{1, \infty}^D=: \frac{1}{r(\Omega)} \quad\text{ and }\quad \lim_{p\to \infty} \left(\lambda_{2,p}^D\right)^{{1}/{p}} = \lambda_{2,\infty}^D=: \frac{1}{r_2(\Omega)},
$$
where
$$r(\Omega)=\sup\{t: \exists B_t(x)\subseteq\Omega\},$$
\begin{center}
    and
\end{center}
$$r_2(\Omega)=\sup\{t: \exists B_t(x_1), B_t(x_2)\subseteq\Omega \text{ such that } B_t(x_1)\cap B_t(x_2)=\emptyset \}.$$

The associated eigenfunctions $v^D_{1,p}$ and $v^D_{2,p}$ also converge (up to a subsequence) to some Lipschitz functions $v^D_{1,\infty}$  and $v^D_{2,\infty}$. These two functions are the first and the second eigenfunctions of a natural viscosity formulation of the eigenvalue problem for the $\infty$-Laplacian (in the sense of Definition \ref{def:viscsol}). 

The Neumann case appears to be more involved. It was investigated in \cite{EKNT, RS2} and similarly to the Dirichlet case, the authors established that the first nontrivial eigenvalues of the $p$-Laplacian $\lbrace\lambda^N_p \rbrace$ satisfy
$$\lim_{p\to \infty}\left(\lambda^N_p\right)^{1/p} = \lambda_\infty^N := \frac{2}{D(\Omega)},$$
where $D(\Omega)$ is the intrinsic diameter of $\Omega$, i.e., the supremum of the geodesic distance between two points of $\Omega$.

However, as in Theorem \ref{secondodavvero}, they were able to show that the sequence of nontrivial first eigenvalues converges to the \emph{first} eigenvalue and the \emph{first} eigenfunction under the additional assumption of convexity on $\Omega$. Whether or not the same applies in the general case is still an open question.

Finally, in \cite{RS}, the authors examined the case of mixed Dirichlet and Robin boundary conditions.
This is a key point in our analysis. In proving our results, we will show that the limit of the second eigenfunctions changes sign, this implies that, on each nodal domain, this limit solves an eigenvalue problem with mixed Dirichlet and Robin boundary conditions.

More precisely,  the authors considered $\Omega$ to be a Lipschitz, bounded, open, and connected set, whose boundary is divided into two parts 
$$
\partial \Omega= \Gamma_1 \cup \Gamma_2.
$$
On $\Gamma_1$ they considered the Dirichlet boundary conditions, while on $\Gamma_2$ the Robin ones. They proved that the sequence of the first mixed eigenvalues converges to 
\begin{equation}\label{rossi}
    \lambda_\infty^{\text{mix}}(\Omega)=\begin{cases}
        \displaystyle{\min_{x\in \Omega} \frac{1}{d(x,\Gamma_1)}} & \text{if } \mathcal{A}=\emptyset,\\[2ex]
        \displaystyle{\min_{x\in \mathcal{A}} \frac{1}{\frac{1}{\beta}+d(x,\Gamma_2)}} & \text{if } \mathcal{A}\neq\emptyset,
    \end{cases}
\end{equation}
where the set $\mathcal{A}$ is defined as follows

$$\mathcal{A}= \left\{x\in \overline{\Omega} : d(x, \Gamma_1)\ge \frac{1}{\beta}+d(x,\Gamma_2)\right\}.$$

The authors also proved that 
\begin{equation}
    \label{mixcarvar}
\lim_{p\to\infty}\left( \lambda_p^{\text{mix}}(\Omega)\right)^{1/p}= \lambda_\infty^{\text{mix}}(\Omega)=\inf_{u\in X} \max\left\{{\norma{\nabla u}_{L^\infty(\Omega)}, \beta \norma{u}_{L^\infty(\Gamma_2)}}\right\},
\end{equation}
where 
\begin{equation}
    \label{def:spacex}
    X=\left\{u\in W^{1,\infty}(\Omega) : u=0 \text{ on } \Gamma_1, \norma{u}_{L^\infty(\Omega)}=1\right\}.
\end{equation}
{
A few final words on the connection between these three problems. From definition \eqref{ess} and Theorem \ref{autolimdiri}, we observe that the following inequality holds
$$\frac{2}{D(\Omega)}\le \lambda_{2,\infty}(\Omega)\le \frac{1}{r_2(\Omega)},$$
where the lower and upper bounds correspond exactly to the Neumann and Dirichlet second
 $\infty$-eigenvalues, respectively. 
Moreover, it can be proved that 

$$
\lim_{\beta\to +\infty}\lambda^\beta_{2,\infty}(\Omega)= \lambda^D_{2,\infty}(\Omega)=\frac{1}{r_2(\Omega)},
$$
while from Definition \eqref{ess}, one can notice that for $\displaystyle{\beta \leq \frac{2}{D(\Omega)}}$, we have $$\displaystyle{\lambda^\beta_{2,\infty}(\Omega)= \lambda_{\infty}^N(\Omega)= \frac{2}{D(\Omega)}}.$$ 

If we analyze the definition \eqref{ess} of $s(\Omega)$, the condition 
$$
\norma{C_{i, s(\Omega)}}_{L^\infty(\partial\Omega)}\le\frac{1}{\beta s(\Omega)}, 
$$
implies that the cones involved in this definition 
satisfy an "infinity Robin boundary conditions"
$$\beta\norma{C_{i, s(\Omega)}}_{L^\infty(\partial\Omega)}\le {\norma{\nabla C_{i, s(\Omega)}}_{L^\infty(\Omega)}}.$$

The latter condition becomes Dirichlet if $\beta$ goes to infinity, and Neumann for a small value of $\beta$.}

The paper is organized as follows: in Section \ref{section_notion} we introduce the notation and the preliminary results that we will need throughout the paper;  Section \ref{limit} is devoted to the study of the limit of the eigenfunction;  while in Sections \ref{geometric}, \ref{varcar_sec} and \ref{infty_eigenproblem}  we prove the main Theorems \ref{autolimdiri}, \ref{var:carteorema} and \ref{secondodavvero}, respectively.

 \section{Notations and Preliminaries}
\label{section_notion}
Throughout this article, $|\cdot|$ will denote the Euclidean norm in $\mathbb{R}^n$, and $\mathcal{H}^k(\cdot)$, for $k\in [0,n)$, will denote the $k$-dimensional Hausdorff measure in $\mathbb{R}^n$.

The following lemma allows us to understand the variational characterization in Theorem \ref{var:carteorema} better, and its proof can be found in \cite{RS}.
\begin{lemma}
Given $f, g \in W^{1,\infty}(\Omega)$, then 
$$
\lim_{p \to \infty} \left( \int_\Omega \abs{f}^p + \int_\Omega \abs{g}^p\right)^{1/p} = \max\left\lbrace \norma{f}_\infty , \norma{g}_\infty \right\rbrace.
$$
\end{lemma}

\subsection{ The viscosity \texorpdfstring{$\infty$}{inf}-eigenvalue problem}

Before proceeding, we give the viscosity formulation of the eigenvalue problem for the $\infty$-Laplacian, see \cite{CIL} for more details.

\begin{definizione}\label{def:viscsol}
  Let us consider the following boundary value problem 
    \begin{equation}
    \label{def:visc}
        \begin{cases}
            F(x, u, \nabla u, \nabla^2u) =0 &\text{ in }\Omega, \\
            G(x,u,\nabla u)=0 &\text{ on }\partial \Omega,
        \end{cases}
    \end{equation}
    where $F:\R^n\times\R\times\R^n\times\R^{n\times n}\to \R$ and $B:\R^n\times \R\times \R^n\to\R$ are two continuous functions.
    \begin{description}
        \item[Viscosity supersolution] A lower semi-continuous function $u$ is a viscosity supersolution to \eqref{def:visc} if, whenever we fix $x_0\in \overline{\Omega}$, for every $\phi \in C^2(\overline{\Omega})$ such that $u(x_0)=\phi (x_0)$ and $x_0$ is a strict minimum in $\Omega$ for $u - \phi$, then
        \begin{itemize}
            \item if $x_0 \in \Omega$, the following holds
            $$
            F\left(x_0,\phi(x_0), \nabla\phi(x_0) ,\nabla^2\phi(x_0)\right) \geq 0
            $$
            \item if $x_0 \in \partial \Omega$, the following holds
            $$
            \max\Set{F\left(x_0,\phi(x_0),\nabla\phi(x_0) ,\nabla^2\phi(x_0)\right), G\left(x_0, \phi(x_0), \nabla\phi(x_0)\right)}\geq 0
            $$
        \end{itemize}
        \item[Viscosity subsolution] An upper semi-continuous function $u$ is a viscosity subsolution to \eqref{def:visc} if,  whenever we fix $x_0\in \overline{\Omega}$, for every $\phi \in C^2(\overline{\Omega})$ such that $u(x_0)=\phi (x_0)$ and  $x_0$ is a strict maximum in $\Omega$ for $u - \phi$, then
        \begin{itemize}
            \item if $x_0 \in \Omega$, the following holds
            $$
            F\left(x_0,\phi(x_0), \nabla\phi(x_0) ,\nabla^2\phi(x_0)\right) \leq 0
            $$
            \item if $x_0 \in \partial \Omega$, the following holds
            $$
            \min\Set{F\left(x_0,\phi(x_0), \nabla\phi(x_0) ,\nabla^2\phi(x_0)\right), G\left(x_0, \phi(x_0), \nabla\phi(x_0)\right)}\leq 0
            $$
        \end{itemize}
         \item[Viscosity solution] A continuous function $u$ is a viscosity solution to \eqref{def:visc} if it is both a super and subsolution.
    \end{description}
\end{definizione}

Hence, we can give the following definition for the $\infty$-eigenvalue problem.
\begin{definizione}
    We say that a function $u\in C(\overline{\Omega})$ is an $\infty$-eigenfunction if there exists $\lambda\neq 0$ such that
    \begin{equation}\label{prob:visc}
        \begin{cases}
            F_{\lambda}(u, \nabla u, \nabla^2 u)=0 & \text{ in } \Omega\\
            G_\beta(u,\nabla u)=0 & \text{ on } \partial \Omega,
        \end{cases}
    \end{equation}

     in the viscosity sense.
    The function appearing in problem \eqref{prob:visc} are the following

\begin{equation*}
F_{\lambda}(u, \nabla u, \nabla^2u)=
\begin{dcases}
{\rm {min}}  \{| \nabla u| - \lambda  |u|, -\Delta_{\infty} u\} & \mbox{in}\ \{u >0\}\cap\Omega\\[1ex]
{\rm {max}} \{  \lambda  |u| -| \nabla u|, -\Delta_{\infty} u\} & \mbox{in}\ \{u <0\}\cap\Omega\\[1ex]
-\Delta_{\infty} u,& \mbox{in} \ \{u =0\}\cap\Omega,
\end{dcases}
\end{equation*}

and 
\begin{equation*}
G_{\beta}(u, \nabla u)=
\begin{dcases}
-{\rm {min}}  \left\{| \nabla u| - \beta u, -\frac{\partial u }{\partial \nu}\right\} & \mbox{in}\ \{u >0\}\cap\partial\Omega\\
-{\rm {max}} \left\{ \beta\abs{u}-| \nabla u|, -\frac{\partial u }{\partial \nu}\right\} & \mbox{in}\ \{u <0\}\cap\partial\Omega\\
\frac{\partial u }{\partial \nu},& \mbox{in} \ \{u =0\}\cap\partial\Omega.
\end{dcases}
\end{equation*}
In this case, the number $\lambda$ is an $\infty$-eigenvalue.
\end{definizione}

\subsection{Properties of the \texorpdfstring{$p$}{f}-Laplace eigenvalues}

We say that a function $u\in W^{1,p}(\Omega)$, $u\neq0$, is a $p$-eigenfunction if there exists a $\lambda>0$ such that
\begin{equation}\label{eq:debole}
  \int _\Omega \abs{\nabla u}^{p-2} \nabla u \nabla \varphi \, dx+ \beta^p \int_{\partial\Omega}\abs{u}^{p-2}u\varphi \, d \mathcal{H}^{n-1} = \lambda\int_{\Omega} \abs{u}^{p-2} u \varphi \, dx, 
\end{equation}
for all $\varphi \in W^{1,p}(\Omega)$,
where $\beta$ is positive. It is possible to define a diverging sequence of eigenvalues, and we need the following definition. 

\begin{definizione}
    Let $E$ be a Banach space and let $A\subset E$ be any closed symmetric set ($v\in A$ implies $-v\in A$). The \emph{genus} $\gamma(A)$ of the set $A$ is defined as the smallest integer $m$ for which there exists a continuous odd mapping $\phi:A\to \R^m\setminus \{0\}$. If no such integer exists, then $\gamma(A)=\infty$. 
\end{definizione}

We define
\begin{equation*}
    \Sigma_k=\{A\subset W^{1,p}(\Omega) \text{ such that }  \gamma(A)\ge k \text{ and } \{v\in A: \norma{v}_p=1\} \text{ is compact }\}.
\end{equation*}
The numbers
\begin{equation}\label{var}
    \lambda_{k,p}(\Omega)= \inf_{A\in \Sigma_k} \sup_{w\in A} \frac{\displaystyle{\int_\Omega \abs{\nabla w}^p\, dx+\beta^p \int_{\partial\Omega} \abs{w}^p\, d\mathcal{H}^{n-1}}}{\displaystyle{\int_\Omega \abs{w}^p\, dx}}
 \end{equation}
are $p$-eigenvalues and create an increasing sequence going to infinity as $k$ diverges. For $p\neq 2$ it is not known whether this sequence contains all the $p$-eigenvalues, but it gives the correct $\lambda_{1,p}(\Omega)$ and $\lambda_{2,p}(\Omega)$, see for instance \cite{spectrump}.

\section{Limit of the second eigenfunctions and its properties}
\label{limit}
As first step, we prove part of Theorem \ref{autolimdiri}. To this end, we recall the definition of $s(\Omega)$
\begin{equation*}
   s(\Omega)=\max \left\{t: \exists x_1, x_2\in \overline{\Omega} \text{ such that }  \abs{x_1-x_2}\ge 2t, \norma{C_{i,t}}_{L^\infty(\partial\Omega)}\le\frac{1}{\beta t}, \, i=1,2 \right\},
\end{equation*} 
where $C_{i,t}$ is the function
$$C_{i,t}(x)=\frac{(t-\abs{x-x_i})_+}{t}.$$
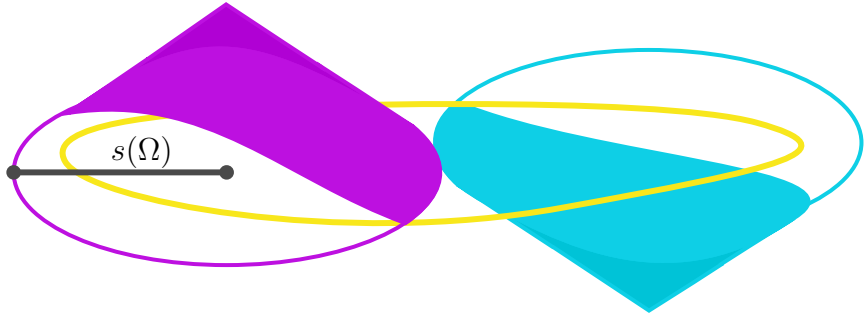
\begin{figure}
    \centering
\tikzset{every picture/.style={line width=0.75pt}} 
\begin{tikzpicture}[x=0.75pt,y=0.75pt,yscale=-1,xscale=1]

\draw [color={rgb, 255:red, 14; green, 207; blue, 230 }  ,draw opacity=1 ][fill={rgb, 255:red, 0; green, 195; blue, 215 }  ,fill opacity=0.66 ][line width=1.5]    (347.5,117.44) .. controls (442.5,180.36) and (347.5,117) .. (443,180.8) .. controls (515.5,134.6) and (442.5,180.36) .. (515.5,134.16) .. controls (457.5,174.2) and (404.5,153.08) .. (347.5,117.44) -- cycle ;
\draw [color={rgb, 255:red, 14; green, 207; blue, 230 }  ,draw opacity=1 ][fill={rgb, 255:red, 14; green, 207; blue, 230 }  ,fill opacity=0.19 ][line width=1.5]    (347.5,117.44) .. controls (332.5,103.8) and (332.5,87.08) .. (347.5,76.08) .. controls (415.5,104.68) and (555.5,109.08) .. (515.5,134.16) .. controls (457.5,174.2) and (404.5,153.08) .. (347.5,117.44) -- cycle ;
\draw  [color={rgb, 255:red, 14; green, 207; blue, 230 }  ,draw opacity=1 ][line width=1.5]  (336.5,96.32) .. controls (336.5,70.57) and (384.07,49.68) .. (442.75,49.68) .. controls (501.43,49.68) and (549,70.57) .. (549,96.32) .. controls (549,122.08) and (501.43,142.96) .. (442.75,142.96) .. controls (384.07,142.96) and (336.5,122.08) .. (336.5,96.32) -- cycle ;
\draw  [color={rgb, 255:red, 248; green, 231; blue, 28 }  ,draw opacity=1 ][line width=2.25]  (229.1,80.03) .. controls (251.03,77.48) and (445.3,72.12) .. (496.69,86.29) .. controls (548.09,100.46) and (508.71,109.28) .. (398.82,129.22) .. controls (288.93,149.15) and (163.1,125.72) .. (151.7,105.84) .. controls (140.29,85.96) and (207.18,82.58) .. (229.1,80.03) -- cycle ;
\draw  [color={rgb, 255:red, 189; green, 16; blue, 224 }  ,draw opacity=1 ][line width=1.5]  (126,111.28) .. controls (126,85.53) and (173.57,64.64) .. (232.25,64.64) .. controls (290.93,64.64) and (338.5,85.53) .. (338.5,111.28) .. controls (338.5,137.04) and (290.93,157.92) .. (232.25,157.92) .. controls (173.57,157.92) and (126,137.04) .. (126,111.28) -- cycle ;
\draw [color={rgb, 255:red, 74; green, 74; blue, 74 }  ,draw opacity=1 ][line width=2.25]    (232.25,111.28) -- (126,111.28) ;
\draw [shift={(126,111.28)}, rotate = 180] [color={rgb, 255:red, 74; green, 74; blue, 74 }  ,draw opacity=1 ][fill={rgb, 255:red, 74; green, 74; blue, 74 }  ,fill opacity=1 ][line width=2.25]      (0, 0) circle [x radius= 2.14, y radius= 2.14]   ;
\draw [shift={(232.25,111.28)}, rotate = 180] [color={rgb, 255:red, 74; green, 74; blue, 74 }  ,draw opacity=1 ][fill={rgb, 255:red, 74; green, 74; blue, 74 }  ,fill opacity=1 ][line width=2.25]      (0, 0) circle [x radius= 2.14, y radius= 2.14]   ;
\draw [color={rgb, 255:red, 189; green, 16; blue, 224 }  ,draw opacity=1 ][fill={rgb, 255:red, 176; green, 1; blue, 212 }  ,fill opacity=0.68 ][line width=1.5]    (153.5,81.36) .. controls (232.5,26.8) and (153.5,81.8) .. (232,26.8) .. controls (326.5,89.8) and (232.5,26.8) .. (325.5,88.4) .. controls (245.5,41.8) and (221.5,32.8) .. (153.5,81.36) -- cycle ;
\draw [color={rgb, 255:red, 189; green, 16; blue, 224 }  ,draw opacity=1 ][fill={rgb, 255:red, 189; green, 16; blue, 224 }  ,fill opacity=0.31 ][line width=1.5]    (153.5,81.36) .. controls (222.5,70.8) and (261.5,114.8) .. (321.5,136.8) .. controls (351.5,115.8) and (335.5,96.8) .. (325.5,88.4) .. controls (245.5,41.8) and (221.5,32.8) .. (153.5,81.36) -- cycle ;

\draw (173,91.6) node [anchor=north west][inner sep=0.75pt]   [align=left] {$s(\Omega)$};

\end{tikzpicture}

    \caption{Definition of $s(\Omega)$.}
    \label{fig:enter-label}
\end{figure}

\begin{prop}\label{autolimdirimezzo}
    Let $\Omega$ be a bounded, {Lipschitz} set, and let $\{\lambda_{2,p}(\Omega)\}$ be the sequence of the second $p$-Laplace eigenvalues. Then, we have

    \begin{equation*}
        \limsup_{p\to \infty} \left(\lambda_{2,p}(\Omega)\right)^\frac{1}{p}\le \frac{1}{s(\Omega)}.
    \end{equation*}
\end{prop}
\begin{proof}
Let us denote by $s:=s(\Omega)$, and let $B_1:=B(x_1, s)$ and $B_2:=B(x_2,s)$ be two disjoint balls with center in $\Omega$.
 We consider the functions
    $$C_i(x)=\frac{(s-\abs{x-x_i})_+}{s}, \quad i=1,2,$$

    and we consider $A=\text{span}\{C_1, C_2\}$. Since the supports of $C_1$ and $C_2$ are disjoint, for every $v\in A$, $v=a_1 C_1+ a_2C_2$ we have

    \begin{equation*}
        \begin{aligned}
            &\int_\Omega \abs{v}^p \, dx= \abs{a_1}^p\int_{B_1\cap \Omega} \abs{C_1}^p+\abs{a_2}^p\int_{B_2\cap \Omega} \abs{C_2}^p ,\\
            &\int_\Omega \abs{\nabla v}^p\, dx= \frac{1}{s^p}(\abs{a_1}^p\abs{B_1\cap \Omega}+\abs{a_2}^p\abs{B_2\cap \Omega}),\\
            &\int_{\partial\Omega} \abs{v}^p\, d\mathcal{H}^{n-1} =\abs{a_1}^p \int_{B_1\cap \partial\Omega} \abs{C_1}^p+\abs{a_2}^p \int_{B_2\cap \partial\Omega} \abs{C_2}^p .
        \end{aligned}
    \end{equation*}
    So, by the variational characterization \eqref{var}, we have
    \begin{equation*}
    \begin{aligned}  &\limsup_{p\to\infty}(\lambda_{2,p})^\frac{1}{p}\le\\ &\limsup_{p\to\infty}\left(\frac{\frac{1}{s^p}(\abs{a_1}^p\abs{B_1\cap \Omega}+\abs{a_2}^p\abs{B_2\cap \Omega})+\beta^p(\abs{a_1}^p \int_{B_1\cap \partial\Omega} \abs{C_1}^p+\abs{a_2}^p \int_{B_2\cap \partial\Omega} \abs{C_2}^p )}{\abs{a_1}^p\int_{B_1\cap \Omega} \abs{C_1}^p+\abs{a_2}^p\int_{B_2\cap \Omega} \abs{C_2}^p}\right)^\frac{1}{p}\le
        \\
        &\limsup_{p\to\infty}\left(\frac{\frac{1}{s^p}  \max\{\abs{B_1\cap \Omega},\abs{B_2\cap \Omega}\}+\beta^p \max\{\int_{B_1\cap \partial\Omega} \abs{C_1}^p,\int_{B_2\cap \partial\Omega} \abs{C_2}^p \}}{  \min\{\int_{B_1\cap \Omega} \abs{C_1}^p,\int_{B_2\cap \Omega} \abs{C_2}^p\}}\right)^\frac{1}{p}=
        \\
    & \frac{\max\left\{\frac 1 s , \beta \norma{C_{1}}_{L^\infty (\partial\Omega)}, \beta \norma{C_{2}}_{L^\infty (\partial\Omega)}\right\}}{\min\left\{\norma{C_i}_{L^\infty(\Omega)}\right\}}\leq \\
        & \frac{1}{s(\Omega)}.
        \end{aligned}
    \end{equation*}
\end{proof}

Proposition \ref{autolimdirimezzo} ensures that, up to a subsequence that we will still denote by $\lambda_{2,p}(\Omega)$,

\begin{equation*}
    \lim_{p\to\infty} \lambda_{2,p}(\Omega)=\overline{\lambda}.
\end{equation*}

A key step in proving the main results is to show the following convergence property of the sequence of second eigenfunctions.
\begin{prop}
\label{uniformconvergence}
    Let $\{u_{p}\}$ be the sequence of the eigenfunctions associated to $\{\lambda_{2,p}(\Omega)\}$.
    Then, there exists a subsequence $\{ u_{p_i}\}$ such that
  \begin{align*}
   & u_{p_i} \to u_{\infty} & \text{ uniformly in } \, \Omega \\
   & \nabla u_{p_i}\to \nabla u_\infty & \text{ weakly in } \, L^q(\Omega), \forall q.
   \end{align*}
\end{prop}
\begin{proof}
Proposition \ref{autolimdirimezzo} shows that the sequence of the second eigenvalues is bounded.
     Thanks to this property, it is possible to show that the sequence $\Set{u_{p}} $ of eigenfunctions associated to $\lambda_{2,p}(\Omega)$, normalized such that $\norma{u_{p}}_{L^{p}}=1$, is uniformly bounded in $W^{1,q}(\Omega)$: indeed, if $q <p$, by H\"older inequality,
\begin{gather*}
    \norma{\nabla u_{p}}_{L^q(\Omega)} \leq \norma{\nabla u_{p}}_{L^{p}(\Omega)} \abs{\Omega}^{\frac{1}{q} - \frac{1}{{p}}} \leq(\lambda_{2, {p}}(\Omega))^{1/{p}} \abs{\Omega}^{\frac{1}{q} - \frac{1}{{p}}} \leq  \frac{1}{s(\Omega)}\abs{\Omega}^\frac{1}{q},\\
    \norma{ u_{p}}_{L^q(\Omega)} \leq \norma{ u_{p}}_{L^{p}(\Omega)} \abs{\Omega}^{\frac{1}{q} - \frac{1}{{p}}} \leq \abs{\Omega}^{\frac{1}{q} - \frac{1}{{p}}} \leq  C,
\end{gather*}
where the constant $C$ is independent of ${p_i}$.

By a classical argument of diagonalization, see for instance \cite{BDM}, we can extract a subsequence,  denoted by $u_{p_i}$ such that
\begin{gather*}
     u_{p_i} \to u_\infty \,  \text{ uniformly}  \implies \Vert u_{p_i}\Vert_{L^{p_i}} \to \norma{u_{\infty}}_{L^{\infty}}, \\
    \nabla u_{p_i} \to \nabla u_\infty \, \text{ weakly in } \, L^q(\Omega), \, \forall q >1.
\end{gather*} 
\end{proof}
Before going on, we show that $u_{\infty}$ is an eigenfunction for the $\infty$-eigenvalue problem.
\begin{teorema}
\label{viscsol}
    Let $\{\lambda_{2,p_i}(\Omega)\}$ be a subsequence of second eigenvalues of the $p_i$-Laplacian with Robin boundary conditions and let $\{u_{p_i}\}$ be the sequence of the associated eigenfunctions. Let $\overline{\lambda}$ be the limit of $\{(\lambda_{2,p_i}(\Omega))^\frac{1}{p_i}\}$ and let $u_\infty$ be the limit of $\{u_{p_i}\}$. Then, $u_\infty$ is a viscosity solution to \eqref{prob:visc}
    \begin{equation*}
        \begin{cases}
            F_{\overline{\lambda}}(u, \nabla u, \nabla^2 u)=0 & \text{ in } \Omega,\\
            G_\beta(u,\nabla u)=0 & \text{ on } \partial \Omega.
        \end{cases}
    \end{equation*}
\end{teorema}

\begin{proof}
    Let us show that $u_\infty$ is a viscosity subsolution. 
\begin{description}
    \item[Case 1:] Let $x_0\in \Omega$ and let $\varphi\in C^2(\Omega)$ such that $u_\infty(x_0)=\varphi(x_0)$ and $u_\infty(x)\le\varphi(x)$.
\begin{description}
    \item[Subcase 1.a:] if $u_\infty(x_0)>0$, we want to show that
    $$\min\Set{\abs{\nabla \varphi(x_0)} - \overline{\lambda} \varphi(x_0), - \Delta_\infty \varphi(x_0)}\le 0.$$
        \begin{itemize}
            \item if $- \Delta_\infty \varphi(x_0)\le 0$ the claim follows. 
            \item if $- \Delta_\infty \varphi(x_0)> 0$, then we have to consider that $u_{p_i}$ uniformly converges to $u_\infty$, so there exists  $B_\rho(x_0)$ such that $u_{p_i}>0$ for all $x\in B_\rho(x_0)$.
            Moreover, $u_{p_i}$ is a viscosity solution to $$-\Delta_{p_i}u_{p_i}=\lambda_{2, p_i}\abs{u_{p_i}}^{p_i-2}u_{p_i}.$$
    
            Notice that $u_{p_i}- \varphi$ has a maximum in $x_i$ and by the uniform convergence $x_i \to x_0$. If we set $\varphi_i(x) =\varphi(x) + c_i$ with $c_i= u_{p_i}(x_i) - \varphi(x_i)$, we can use $\varphi_i$ in the definition of viscosity subsolution, obtaining
            \begin{equation*}
                -\abs{\nabla \varphi_i(x_i)}^{p_i-2} \hspace{-3pt}\Delta \varphi_i(x_i)- (p_i-2) \abs{\nabla \varphi_i(x_i)}^{p_i-4}\hspace{-.6pt}\Delta_\infty \varphi (x_i)\le \hspace{-1mm}\lambda_{2, p_i} \abs{\varphi_i(x_i)}^{p_i-2}\varphi_i(x_i).
            \end{equation*}
            To show that  $$\abs{\nabla\varphi(x_0)}-\overline{\lambda}\varphi(x_0)\le 0,$$
            we can divide both sides of the inequality for $(p_i-2)\abs{\varphi_i(x_i)}^{p_i-4}$, and raise everything to the power $1/p_i$ (the left-hand side is positive since $- \Delta_\infty \varphi(x_0)> 0$), we get
            \begin{equation*}
               \left( \frac{-\abs{\nabla \varphi_i(x_i)}^{2} \Delta \varphi_i(x_i)}{(p_i-2)}- \Delta_\infty \varphi (x_i)\right)^{\frac{1}{p_i}}\le \left(\frac{\lambda_{2, p_i} \abs{\varphi_i(x_i)}^{p_i-2}\varphi_i(x_i)}{(p_i-2) \abs{\nabla\varphi_i(x_i)}^{p_i-4}}\right)^{\frac{1}{p_i}},
            \end{equation*}
            and we get the claim. Let us observe that we need to compute the power $1/p_i$ in order to properly understand the behaviour of the right-hand side.        
        \end{itemize}
\item[Subcase 1.b:] if  $u_\infty(x_0)=0$, we want to show that $$-\Delta_\infty \varphi(x_0)\le 0.$$
We argue as before obtaining
\begin{equation*}
    -\Delta_\infty\varphi(x_i)\le \frac{\abs{\nabla\varphi(x_i)}^2\Delta\varphi(x_i)}{p_i-2}+ \left(\frac{\lambda_{2,p_i}^{\frac{1}{p_i-4}}\abs{\varphi_i(x_i)}}{\abs{\nabla\varphi(x_i)}}\right)^{p_i-4}\frac{\abs{\varphi_i(x_i)}^2\varphi_i(x_i)}{p_i-2},
\end{equation*}
that for $p_i\to\infty$ gives
$$-\Delta_\infty\varphi(x_0)\le 0.$$
\item[Subcase 1.c:] if  $u_\infty(x_0)<0$, we want to show that
    $$\max\Set{-\abs{\nabla \varphi(x_0)} - \overline{\lambda}\varphi(x_0), - \Delta_\infty \varphi(x_0)}\le 0.$$

With the same notation as before, we can use $\varphi_i$ in the definition of viscosity subsolution, obtaining 
\begin{equation*}
    -\abs{\nabla \varphi_i(x_i)}^{p_i-2} \Delta \varphi_i(x_i)- (p_i-2) \abs{\nabla \varphi_i(x_i)}^{p_i-4}\Delta_\infty \varphi_i (x_i)\le\lambda_{2,p_i} \abs{\varphi_i(x_i)}^{p_i-2}\varphi_i(x_i).
\end{equation*}
Now, dividing by $(p_i-2) \abs{\nabla \varphi_i(x_i)}^{p_i-4}$, we obtain
\begin{equation}
\label{f1}
   - \Delta_\infty \varphi_i(x_i) - \frac{\abs{\nabla \varphi_i(x_i)}^{2 } \Delta \varphi_i(x_i)}{ p_i-2} \leq \frac{ \abs{\nabla \varphi_i(x_i)}^4\varphi_i(x_i)}{(p_i-2) \abs{\varphi_i(x_i)}^2}	\left(\frac{\lambda_{2,p_i}^{1/p_i} \abs{\varphi_i(x_i)}}{ \abs{\nabla \varphi_i(x_i)}}\right)^{p_i}
\end{equation}

  This gives us
$
-\abs{\nabla \varphi(x_0)}-\overline{\lambda} \varphi(x_0)\leq 0 $ since, otherwise, the right-hand side of \eqref{f1} would go to minus infinity, in contradiction with the fact that $\varphi \in C^2(\Omega)$. Moreover 
$- \Delta_\infty \varphi(x_0)\leq 0
$, just taking the limit.

\end{description}
    \item[Case 2:] Let $x_0\in \partial \Omega$, let $\varphi\in C^2$ such that $u_\infty(x_0)=\varphi(x_0)$ and $u_\infty(x)\le\varphi(x)$. Let $x_i$ the maximum of $u_{p_i}-\varphi$, and let us consider $\varphi_i$ as before. As already observed, $x_i\to x_0$.
If $x_i\in \Omega$ for infinitely many $i$, then, Case 1 shows that

$$F_{\overline{\lambda}}(u, \nabla u, \nabla^2 u)\le 0, \quad \Rightarrow \quad \min\{F_{\overline{\lambda}}(u, \nabla u, \nabla^2 u), G_\beta(u,\nabla u)\}\le 0.$$

If $x_i\in \partial\Omega$ for infinitely many $i$, then we have
\begin{equation}
    \label{machesoio}
\abs{\nabla\varphi_i(x_i)}^{p_i-2}\frac{\partial\varphi_i(x_i)}{\partial\nu}+\beta^{p_i}\abs{\varphi_i(x_i)}^{p_i-2}\varphi_i(x_i)\le 0.
\end{equation}
\begin{description}
    \item[Subcase 2.a:] if $u_\infty(x_0)>0$, by \eqref{machesoio}, we have that both $|\nabla\varphi_i(x_i)| >0$ and \\    
    $\displaystyle{\frac{\partial \varphi_i(x_i)}{\partial\nu} <0}$. Moreover, from \eqref{machesoio} we also have
    $$\left(\dfrac{\beta^{p_i}\varphi_i(x_i)^{p_i-1}}{\abs{\nabla\varphi_i(x_i)}^{p_i-2}}\right)^\frac{1}{p_i}\le \left(-\frac{\partial \varphi_i(x_i)}{\partial \nu}\right)^\frac{1}{p_i}$$
    and if we take the limit we get
$$-\min\left\{|\nabla\varphi(x_0)|-\beta \varphi(x_0), -\frac{\partial \varphi(x_0)}{\partial\nu}\right\}\le 0.$$
\item[Subcase 2.b:] if  $u_\infty(x_0)=0$,  then $\displaystyle{\frac{\partial \varphi(x_0)}{\partial \nu}\leq 0}$, otherwise  we should have 
   $$1 \leq \left(\dfrac{-\beta^{p_i}|\varphi_i(x_i)|^{p_i-2}\varphi_i(x_i)}{\abs{\nabla\varphi_i(x_i)}^{p_i-2}}\right)^\frac{1}{p_i}\left(\frac{\partial \varphi_i(x_i)}{\partial \nu}\right)^{-\frac{1}{p_i}} \to 0,$$
   which is a contradiction.

   \item[Subcase 2.c:] if $u_\infty(x_0)<0$, we have to distinguish the case 
 $\displaystyle{\frac{\partial\varphi_i(x_i)}{\partial\nu}\le0}$ that immediately gives
 $$-\max\left\{-|\nabla\varphi(x_0)|-\beta \varphi(x_0), -\frac{\partial \varphi(x_0)}{\partial\nu}\right\}\le 0,$$
 
 and the case $\displaystyle{\frac{\partial\varphi_i(x_i)}{\partial\nu}>0}$, that implies
    $$\left(\dfrac{-\beta^{p_i}|\varphi_i(x_i)|^{p_i-2}\varphi_i(x_i)}{\abs{\nabla\varphi_i(x_i)}^{p_i-2}}\right)^\frac{1}{p_i}\ge \left(\frac{\partial \varphi_i(x_i)}{\partial \nu}\right)^\frac{1}{p_i},$$
    and if we take the limit we get
$$-\max\left\{-|\nabla\varphi(x_0)|-\beta \varphi(x_0), -\frac{\partial \varphi(x_0)}{\partial\nu}\right\}\le 0.$$
\end{description} 
\end{description}

   Let us show that $u_\infty$ is a viscosity supersolution. 
\begin{description}
    \item[Case 3:] Let $x_0\in \Omega$ and let $\varphi\in C^2(\Omega)$ such that $u_\infty(x_0)=\varphi(x_0)$ and $u_\infty(x)\ge\varphi(x)$.
\begin{description}
    \item[Subcase 3.a:] if  $u_\infty(x_0)>0$, we want to show that
    $$\min\Set{\abs{\nabla \varphi(x_0)} - \overline{\lambda} \varphi(x_0), - \Delta_\infty \varphi(x_0)}\ge 0.$$

With the same notation as before, we can use $\varphi_i$ in the definition of viscosity supersolution, obtaining 
\begin{equation*}
    -\abs{\nabla \varphi_i(x_i)}^{p_i-2} \Delta \varphi_i(x_i)- (p_i-2) \abs{\nabla \varphi_i(x_i)}^{p_i-4}\Delta_\infty \varphi_i (x_i)\ge\lambda_{2,p_i} \abs{\varphi_i(x_i)}^{p_i-2}\varphi_i(x_i).
\end{equation*}
Now, dividing by $(p_i-2) \abs{\nabla \varphi_i(x_i)}^{p_i-4}$, we obtain
\begin{equation}
\label{f2}
   - \Delta_\infty \varphi_i(x_i) - \frac{\abs{\nabla \varphi_i(x_i)}^{2 } \Delta \varphi_i(x_i)}{ p_i-2} \geq \frac{ \abs{\nabla \varphi_i(x_i)}^4\varphi_i(x_i)}{(p_i-2) \abs{\varphi_i(x_i)}^2}	\left(\frac{\lambda_{2,p_i}^{1/p_i} \varphi_i(x_i)}{ \abs{\nabla \varphi_i(x_i)}}\right)^{p_i}
\end{equation}

  This gives us
$
\abs{\nabla \varphi(x_0)}-\overline{\lambda}\varphi(x_0)\geq 0 $ since, otherwise, the right-hand side of \eqref{f2} would go to infinity, in contradiction with the fact that $\varphi \in C^2(\Omega)$. Moreover 
$- \Delta_\infty \varphi(x_0)\geq 0
$, just taking the limit.
\item[Subcase 3.b:] if  $u_\infty(x_0)=0$, we want to show that $$-\Delta_\infty \varphi(x_0)\ge 0.$$
We argue as before obtaining
\begin{equation*}
    -\Delta_\infty\varphi_i(x_i)\ge \frac{\abs{\nabla\varphi_i(x_i)}^2\Delta\varphi_i(x_i)}{p_i-2}+ \left(\frac{\lambda_{2,p_i}^{\frac{1}{p_i-4}}\abs{\varphi_i(x_i)}}{\abs{\nabla\varphi_i(x_i)}}\right)^{p_i-4}\frac{\abs{\varphi_i(x_i)}^2\varphi_i(x_i)}{p_i-2},
\end{equation*}
that for $p_i\to\infty$ gives
$$-\Delta_\infty\varphi(x_0)\ge 0.$$
\item[Subcase 3.c:] if $u_\infty(x_0)<0$, we want to show that
    $$\max\Set{-\abs{\nabla \varphi(x_0)} - \overline{\lambda} \varphi(x_0), - \Delta_\infty \varphi(x_0)}\ge 0.$$
        \begin{itemize}
            \item if $- \Delta_\infty \varphi(x_0)\ge 0$ the claim follows. 
            \item if $- \Delta_\infty \varphi(x_0)< 0$, then we have to consider that $u_{p_i}$ uniformly converges to $u_\infty$, so there exists  $B_\rho(x_0)$ such that $u_{p_i}<0$ for all $x\in B_\rho(x_0)$.
            Moreover, $u_{p_i}$ is a viscosity solution to $$-\Delta_{p_i}u_{p_i}=\lambda_{2, p_i}\abs{u_{p_i}}^{p_i-2}u_{p_i}.$$
    
            Notice that $u_{p_i}- \varphi$ has a minimumimum in $x_i$ and by the uniform convergence $x_i \to x_0$. If we set $\varphi_i(x) =\varphi(x) + c_i$ with $c_i= u_{p_i}(x_i) - \varphi(x_i)$, we can use $\varphi_i$ in the definition of viscosity supersolution, obtaining
            \begin{equation*}
                -\abs{\nabla \varphi_i(x_i)}^{p_i-2} \hspace{-1.75pt} \Delta \varphi_i(x_i)- (p_i-2) \abs{\nabla \varphi_i(x_i)}^{p_i-4}\hspace{-1.75pt}\Delta_\infty \varphi (x_i)\hspace{-1mm}\ge \lambda_{2, p_i} \abs{\varphi_i(x_i)}^{p_i-2}\varphi_i(x_i).
            \end{equation*}
             We can divide both sides of the inequality for $(p_i-2)\abs{\varphi_i(x_i)}^{p_i-4}$, and raise everything to the power $1/p_i$ (the left-hand side is positive since $- \Delta_\infty \varphi(x_0)< 0$), we get
            \begin{equation*}
               \left( \frac{\abs{\nabla \varphi_i(x_i)}^{2} \Delta \varphi_i(x_i)}{(p_i-2)}+ \Delta_\infty \varphi_i (x_i)\right)^{\frac{1}{p_i}}\le \left(-\frac{\lambda_{2, p_i} \abs{\varphi_i(x_i)}^{p_i-2}\varphi_i(x_i)}{(p_i-2) \abs{\nabla\varphi_i(x_i)}^{p_i-4}}\right)^{\frac{1}{p_i}},
            \end{equation*}
            that gives
            $$-\abs{\nabla\varphi(x_0)}-\overline{\lambda}\varphi(x_0)\ge 0.$$
        \end{itemize}

\end{description}
    \item[Case 4:] Let $x_0\in \partial \Omega$, let $\varphi\in C^2$ such that $u_\infty(x_0)=\varphi(x_0)$ and $u_\infty(x)\ge\varphi(x)$. Let $x_i$ be the minimum of $u_{p_i}-\varphi$, and let us build $\varphi_i$ as before. As already observed, $x_i\to x_0$.  If $x_i\in \Omega$ for infinitely many $i$, then, Case 3 shows that

$$F_{\overline{\lambda}}(u, \nabla u, \nabla^2 u)\ge 0, \quad \Rightarrow \quad \max\{F_{\overline{\lambda}}(u, \nabla u, \nabla^2 u), G_\beta(u,\nabla u)\}\ge 0.$$

If $x_i\in \partial\Omega$ for infinitely many $i$, then we have
\begin{equation}
    \label{machesoio2}
\abs{\nabla\varphi_i(x_i)}^{p_i-2}\frac{\partial\varphi_i(x_i)}{\partial\nu}+\beta^{p_i}\abs{\varphi_i(x_i)}^{p_i-2}\varphi_i(x_i)\ge 0.
\end{equation}
\begin{description}
 \item[Subcase 4.a:] if $u(x_0)>0$, we have to distinguish the case 
 $\frac{\partial\varphi_i(x_i)}{\partial\nu}\ge0$ that immediately gives
 $$-\min\left\{|\nabla\varphi(x_0)|-\beta \varphi(x_0), -\frac{\partial \varphi(x_0)}{\partial\nu}\right\}\ge 0,$$
 
 and the case $\frac{\partial\varphi_i(x_i)}{\partial\nu}<0$,
    $$\left(\dfrac{\beta^{p_i}|\varphi_i(x_i)|^{p_i-2}\varphi_i(x_i)}{\abs{\nabla\varphi_i(x_i)}^{p_i-2}}\right)^\frac{1}{p_i}\ge \left(-\frac{\partial \varphi_i(x_i)}{\partial \nu}\right)^\frac{1}{p_i}$$
    and if we take the limit we get
$$-\min\left\{|\nabla\varphi(x_0)|-\beta \varphi(x_0), -\frac{\partial \varphi(x_0)}{\partial\nu}\right\}\ge 0.$$
    
\item[Subcase 4.b:] if  $u(x_0)=0$,  then $\frac{\partial \varphi(x_0)}{\partial \nu}\geq 0$, otherwise  we should have 
   $$1 \leq \left(\dfrac{\beta^{p_i}|\varphi_i(x_i)|^{p_i-2}\varphi_i(x_i)}{\abs{\nabla\varphi_i(x_i)}^{p_i-2}}\right)^\frac{1}{p_i}\left(\dfrac{1}{-\frac{\partial \varphi_i(x_i)}{\partial \nu}}\right)^\frac{1}{p_i} \to 0,$$
   which is a contradiction.

\item[Subcase 4.c:] if $u(x_0)<0$, by \eqref{machesoio2}, we have that both $|\nabla\varphi_i(x_i)| >0$ and $\frac{\partial \varphi_i(x_i)}{\partial\nu} >0$. Moreover, from \eqref{machesoio} we also have
    $$\left(\dfrac{-\beta^{p_i}\varphi_i(x_i)^{p_i-1}}{\abs{\nabla\varphi_i(x_i)}^{p_i-2}}\right)^\frac{1}{p_i}\le \left(\frac{\partial \varphi_i(x_i)}{\partial \nu}\right)^\frac{1}{p_i}$$
    and if we take the limit we get
$$-\max\left\{-|\nabla\varphi(x_0)|-\beta \varphi(x_0), -\frac{\partial \varphi(x_0)}{\partial\nu}\right\}\ge 0.$$
\end{description} 
\end{description}
\end{proof}

To conclude the proof of Theorem \ref{autolimdiri}, in the case $\beta> 2/D(\Omega)$, we need to prove a sign condition on the limit of the eigenfunctions. This is contained in Lemma \ref{segnpiu}. 


\begin{lemma}\label{segnpiu}
    Let $\beta>\frac{2}{D(\Omega)}$. 
    The limit $u_\infty$ of the second eigenfunction of the $p$-Laplace operator with Robin boundary conditions changes sign. 
\end{lemma}
\begin{proof}
Let us denote by 
$$ \Omega_i^+=\{x\in \Omega : u_{p_i}(x)>0\}, \quad \Omega_i^-=\{x\in \Omega : u_{p_i}(x)<0\},$$
then, if we choose $\varphi=u_{p_i}^+$ in the weak formulation \eqref{eq:debole}, 
\begin{align*}
\left(\lambda_{2,p_i}(\Omega)\right)^{\frac{1}{p_i}}&=\left(\frac{\int_{\Omega_i^+} \abs{\nabla u_{p_i}}^{p_i}+\beta^{p_i} \int_{\Omega_i^+\cap\partial\Omega}\abs{u_{p_i}}^{p_i}}{\int_{\Omega_i^+}\abs{u_{p_i}}^{p_i}}\right)^{\frac{1}{p_i}}\\&\ge \left(\lambda_{1,p_i}( \Omega_i^+)\right)^{\frac{1}{p_i}}
\\&\ge \left(\lambda_{1,p_i}( (\Omega_i^+)^{\sharp})\right)^{\frac{1}{p_i}}
\\&\ge \frac{p_i-1}{p_i} \frac{\omega_n^{\frac{1}{n}} }{\abs{\Omega_i^+}^\frac{1}{np_i} \left(\omega_n^{\frac{1}{n}}\beta^{-\frac{p_i}{p_i-1}} + \abs{\Omega_i^+}^{\frac{1}{n}}\right)^{\frac{p_i-1}{p_i}}}
\\&\geq \frac{p_i-1}{p_i} \frac{\omega_n^{\frac{1}{n}}}{\omega_n^{\frac{1}{n}}\beta^{-\frac{p_i}{p_i-1}} + \abs{\Omega_i^+}^{\frac{1}{n}}} 
\end{align*}
where the next-to-last inequality is proved in \cite[Proposition 3.1]{DLG}.
Hence

\begin{align*}
\abs{\Omega_i^+} \geq  \omega_n\left(\frac{p_i-1}{p_i\left(\lambda_{p_i}\right)^{\frac{1}{p_i}}}-\frac{1}{\beta^{\frac{p_i}{p_i-1}}}\right)^n .
\end{align*}
Thus, by
\begin{equation*}
    \limsup_k \Omega_{i_k}^+:= \bigcap_{h=1}^\infty\bigcup_{k=h}^\infty \Omega^+_{i_k} \subset \left\{u_\infty\ge 0\right\},
\end{equation*}
 and by Proposition \ref{autolimdirimezzo}, we have
\begin{equation*}
    \abs{\left\{u_\infty\ge 0\right\}} \geq \abs{ \limsup_k \Omega_{i_k}^+} \geq  \limsup_k \abs{\Omega_{i_k}^+} \geq \omega_n \left(s(\Omega)-\frac{1}{\beta}\right)^n
\end{equation*}
It remains to show that
\begin{equation}
    \label{claim}
    s(\Omega)-\frac{1}{\beta}  >0
\end{equation}
when $\beta > \frac{2}{D(\Omega)}$.

Firstly, let us observe that $1/\beta$ belongs to the set
\begin{equation*}
  \left\{t: \exists x_1, x_2\in \overline{\Omega} \text{ such that }  \abs{x_1-x_2}\ge 2t, \norma{C_{i,t}}_{L^\infty(\partial\Omega)}\le\frac{1}{\beta t} , \, i=1,2\right\},
\end{equation*}
so by the definition of $s(\Omega)$, it holds $s(\Omega)\ge 1/\beta$.
We claim that $s(\Omega)>1/\beta$. 

Let $\Omega_t= \{x\in \Omega : d(x,\partial\Omega)>t\}$ be the inner parallel set of $\Omega$, we define the function $h$ by 

$$h(t):=D(\Omega_t)-2\left(\frac{1+\beta t}{\beta}\right),$$

$h$ is a continuous function satisfying $h(0)>0$, so there exists $\overline{t}>0$ such that $h(\overline{t})>0.$ We want to show that
$$s(\Omega)\ge \frac{1+\beta \overline{t}}{\beta}>\frac{1}{\beta}.$$

To this end, we consider two points $x_1, x_2$ such that $|x_1-x_2|=D(\Omega_{\overline{t}})>2\left(\frac{1+\beta \overline{t}}{\beta}\right)$, and the functions
$$C_{i, \frac{1+\beta\overline{t}}{\beta}}=\frac{\left(\frac{1+\beta\overline{t}}{\beta}-\abs{x-x_i}\right)_+}{\frac{1+\beta\overline{t}}{\beta}},$$
it holds

$$\norma{C_{i, \frac{1+\beta\overline{t}}{\beta}}}_{L^\infty(\partial\Omega)}=\frac{1}{1+\beta\overline{t}}.$$
By the definition \eqref{ess} of $s(\Omega)$, the claim \eqref{claim} holds.

Similiar estimates hold for $\abs{\{u_\infty\le 0\}}$.

Let us observe that the function $u_\infty$ cannot be non-negative or non-positive. Indeed, if $u_\infty$ were non-negative, Theorem \ref{viscsol}, would imply that $u_\infty$ is a non-negative viscosity supersolution to $-\Delta_\infty u=0$. The Harnack inequality (see \cite{har2,harnack}) would imply $u_\infty>0$ in $\Omega$, against the fact that 
$\abs{\{u_\infty\le 0\}}>0$
\end{proof}

Once we have obtained Lemma \ref{segnpiu}, it is natural to ask whether a similar property holds in the case  $\beta \leq 2/D(\Omega)$.
In the following Lemma, we provide a positive answer for $\beta < 2/D(\Omega)$; nut it is still open if the same holds for  the case $\beta = 2/D(\Omega)$.

\begin{lemma}
    Let $\beta<\frac{2}{D(\Omega)}$, then the limit of the second eigenfunction of the $p$-Laplace operator with Robin boundary conditions satisfies
$$-\inf u_\infty =\sup{u_\infty}.$$
\end{lemma}
\begin{proof}
    Let us denote by
    $$u_{p_i}^+=\max\{u_{p_i}, 0\}, \quad u_{p_i}^-= \max\{-u_{p_i}, 0\}.$$
    We can rewrite the weak formulation \eqref{eq:debole} with $\varphi\equiv1$ in this way

    \begin{equation*}
    \begin{aligned}
         \beta^{p_i}\int_{\partial\Omega}\abs{u_{p_i}^+}^{p_i-2}u_{p_i}^+ \, d \mathcal{H}^{n-1}-\beta^{p_i}\int_{\partial\Omega}\abs{u_{p_i}^-}^{p_i-2}u_{p_i}^- \, d \mathcal{H}^{n-1} = \\
        \lambda_{2, p_i}(\Omega)\int_{\Omega} \abs{u_{p_i}^+}^{{p_i}-2} u_{p_i}^+  \, dx-\lambda_{2, p_i}(\Omega)\int_{\Omega} \abs{u_{p_i}^-}^{{p_i}-2} u_{p_i}^-  \, dx,
    \end{aligned} 
    \end{equation*}
    If we rearrange the terms and we let ${p_i}\to + \infty$, we obtain
 
\begin{equation*}
    \max\{\beta \norma{u_\infty^+}_{L^\infty(\partial\Omega)}, \lambda_{2,\infty} (\Omega)\norma{u_\infty^-}_{L^\infty(\Omega)}\}=\max\{\beta\norma{u_\infty^-}_{L^\infty(\partial\Omega)},\lambda_{2,\infty}(\Omega)\norma{u_\infty^+}_{L^\infty(\Omega)}\}. 
\end{equation*}
Moreover, if $\displaystyle{\beta\leq\frac{2}{D(\Omega)}}$
\begin{equation}
\label{sminbeta}
\frac{2}{D(\Omega)} = \lambda_{\infty}^N(\Omega) = \lim_{p_i} (\lambda_{{p_i}}^N(\Omega) )^{\frac{1}{{p_i}}}\leq\lim_{p_i}  (\lambda_{2,{p_i}}(\Omega))^{\frac{1}{{p_i}}}\leq \frac{1}{s(\Omega)} =  \frac{2}{D(\Omega)},\end{equation}
 
 as $D(\Omega)/2$ is admissible in the definition \eqref{ess} of $s(\Omega)$. 
 Indeed, if we consider the two endpoints of a diameter of the set $\Omega$, we have
 $$
 \abs{x_1-x_2} = 2 \frac{D(\Omega)}{2} \quad\text{ and } \quad \norma{C_{i,\frac{D(\Omega)}{2}}}_{L^\infty(\partial \Omega)} \leq 1\leq \frac{2}{\beta D(\Omega)}.
 $$

 Four cases can occur:

\begin{description}
    \item[1.] We have $$\begin{aligned}
        &\max\{\beta \norma{u_\infty^+}_{L^\infty(\partial\Omega)}, \lambda_{2,\infty} (\Omega)\norma{u_\infty^-}_{L^\infty(\Omega)}\}=\beta \norma{u_\infty^+}_{L^\infty(\partial\Omega)}=\\
    &\max\{\beta\norma{u_\infty^-}_{L^\infty(\partial\Omega)},\lambda_{2,\infty}(\Omega)\norma{u_\infty^+}_{L^\infty(\Omega)}\}=\beta\norma{u_\infty^-}_{L^\infty(\partial\Omega)},
    \end{aligned}$$  and, since $\beta < 2 /D(\Omega)$, the last would imply
$$\lambda_{2,\infty}(\Omega)\norma{u^-_\infty}_{L^\infty(\Omega)}\le\beta \norma{u_\infty^+}_{L^\infty(\partial\Omega)}= \beta\norma{u_\infty^-}_{L^\infty(\partial\Omega)}< \lambda_{2,\infty}(\Omega)\norma{u_\infty^-}_{L^\infty(\partial\Omega)},$$
that is absurd.
\item[2.]  We have
$$\begin{aligned}
        &\max\{\beta \norma{u_\infty^+}_{L^\infty(\partial\Omega)}, \lambda_{2,\infty} (\Omega)\norma{u_\infty^-}_{L^\infty(\Omega)}\}=\beta \norma{u_\infty^+}_{L^\infty(\partial\Omega)}=\\
    &\max\{\beta\norma{u_\infty^-}_{L^\infty(\partial\Omega)},\lambda_{2,\infty}(\Omega)\norma{u_\infty^+}_{L^\infty(\Omega)}\}=\lambda_{2,\infty}(\Omega)\norma{u_\infty^+}_{L^\infty(\Omega)},
    \end{aligned}$$
     and similarly to the first case, we get an absurd
    $$\beta \norma{u_\infty^+}_{L^\infty(\partial\Omega)}=\lambda_{2,\infty}(\Omega)\norma{u_\infty^+}_{L^\infty(\Omega)}>\beta\norma{u_\infty^+}_{L^\infty(\Omega)}.$$

    \item[3.]  We have $$\begin{aligned}
        &\max\{\beta \norma{u_\infty^+}_{L^\infty(\partial\Omega)}, \lambda_{2,\infty}(\Omega) \norma{u_\infty^-}_{L^\infty(\Omega)}\}=\lambda_{2,\infty} (\Omega)\norma{u_\infty^-}_{L^\infty(\Omega)}=\\
    &\max\{\beta\norma{u_\infty^-}_{L^\infty(\partial\Omega)},\lambda_{2,\infty}(\Omega)\norma{u_\infty^+}_{L^\infty(\Omega)}\}=\beta\norma{u_\infty^-}_{L^\infty(\partial\Omega)},
    \end{aligned}$$
    but, arguing as before, this would imply $\norma{u_\infty^-}_{L^\infty(\partial\Omega)}>\norma{u_\infty^-}_{L^\infty(\Omega)}$.
    \item[4] The only case that can occur is 
    $$\norma{u_\infty^-}_{L^\infty(\Omega)}=\norma{u_\infty^+}_{L^\infty(\Omega)},$$
    that is the claim.
\end{description}

\end{proof}

\section{The geometric characterization}
\label{geometric}
Now we complete the proof of Theorem \ref{autolimdiri}.
\begin{proof}[Proof of Theorem \ref{autolimdiri}] Proposition \ref{autolimdirimezzo} gives
$$
\limsup_p (\lambda_{2,p}(\Omega))^\frac{1}{p}\leq \frac{1}{s(\Omega)},
$$
so it remains to prove

$$
\liminf_p (\lambda_{2,p}(\Omega))^\frac{1}{p}\geq \frac{1}{s(\Omega)}.
$$

We distinguish the cases  $\beta \leq 2/D(\Omega)$ and  $\beta > 2/D(\Omega)$.
In the case   $\beta \leq 2/D(\Omega)$ the thesis is given by \eqref{sminbeta}. 

    To prove the claim  in the case $\beta > 2/D(\Omega)$, we use $(u_{p})_+$, normalized so that $\norma{(u_{p})_+}_{L^p(\Omega^+_p)}=1$, in the weak formulation \eqref{eq:debole} with $\lambda=\lambda_{2,p}(\Omega)$, obtaining
    \[
    \liminf_p (\lambda_{2,p}(\Omega))^\frac{1}{p} = \liminf_p \left(\frac{\displaystyle{\int_{\Omega_p^+} \abs{\nabla u_p}^p+ \beta^p \int_{\Omega^+_p\cap \partial \Omega} u^p}}{\displaystyle{\int_{\Omega^+_p} u^p}}\right)^{\frac{1}{p}}.
    \]
    By the uniform convergence of $u_{p}$, stated in Proposition \ref{uniformconvergence}, and by Lemma \ref{segnpiu}, we have that both $u^+_{p}$ and $u^-_{p}$ uniformly converge to $u^+_{\infty}\ne 0$ and $u^-_{\infty}\ne 0$, respectively (for instance write $2u^+_{p}= \abs{u_p}+u_p$).
    Since $(\nabla u_p)\chi_{\Omega^+_p}$ and  $(\nabla u_p)\chi_{\Omega^-_p}$  are equi-bounded in $L^q(\Omega)$ for all $q$ (see Proposition \ref{uniformconvergence}), up to a subsequence they converge to $(\nabla u_\infty)\chi_{\Omega^+_\infty}$ and  $(\nabla u_\infty)\chi_{\Omega^-_\infty}$ weakly in $L^q(\Omega)$.
Hence, the following inequality holds
\begin{align*}
    \frac{\norma{\nabla u_\infty}_{L^q(\Omega^+_\infty)}}{\norma{u_\infty}_{L^q(\Omega^+_\infty)}}&\le \liminf_{p \to \infty} \frac{\norma{\nabla u_p}_{L^q(\Omega^+_\infty)}}{\norma{u_p}_{L^q(\Omega^+_\infty)}}\le
     \liminf_{p \to \infty} \frac{\norma{\nabla u_p}_{L^p(\Omega^+_\infty)}}{\norma{u_p}_{L^q(\Omega^+_\infty)}}  \abs{\Omega}^{\frac{1}{q}-\frac{1}{p}}\\ &\le  \frac{\abs{\Omega}^{\frac{1}{q}}}{\norma{u_\infty}_{L^q(\Omega^+_\infty)}} \liminf_{p \to \infty} \left( \lambda_{2,p}(\Omega)\right)^{\frac{1}{p}}.
\end{align*}
Letting $q\to\infty$, we obtain $$\norma{\nabla u_\infty}_{L^\infty(\Omega^+_\infty)}\le \liminf_{p \to \infty} \left( \lambda_{2,p}(\Omega)\right)^{\frac{1}{p}} .$$

Similarly
$$
\beta \norma{ u_p}_{L^q(\partial\Omega^+_\infty \cap \partial \Omega)} \leq \beta \norma{\ u_p}_{L^p(\partial \Omega^+_\infty \cap \partial\Omega)} \abs{\partial\Omega}^{\frac{1}{q} - \frac{1}{p}} \leq(\lambda_{2,p}(\Omega))^{1/p} \abs{\partial\Omega}^{\frac{1}{q} - \frac{1}{p}}\leq  C,
$$
gives us
$$
\beta \norma{u_\infty}_{L^\infty(\partial \Omega^+_\infty \cap \partial\Omega)}\le \liminf_{p \to \infty} \left( \lambda_{2,p}(\Omega)\right)^{\frac{1}{p}},
$$
  hence
    $$
     \liminf_p (\lambda_{2,p}(\Omega))^\frac{1}{p} \geq \inf_{\substack{w \in W^{1,\infty}(\Omega) \\
     w|_{\partial\Omega^+_\infty \cap \Omega }=0}}\frac{\max\{ \norma{\nabla w}_\infty, \beta \norma{ w}_{L^\infty(\partial \Omega^+_\infty \cap \partial\Omega)}\}}{\norma{u}_\infty} = \lambda_\infty^{\text{mix}}(\Omega^+_\infty),
    $$
    where $\lambda_\infty^{\text{mix}}(\Omega^+_\infty)$ is the one defined in \cite{RS2} and reported in formulas \eqref{rossi} and \eqref{mixcarvar}.
    
    An analogous argument for $\lambda_\infty^{\text{mix}}(\Omega^-_\infty)$, brings us to 
    $$
     \liminf_p (\lambda_{2,p}(\Omega))^\frac{1}{p} \geq\max\{\lambda_\infty^{\text{mix}}(\Omega^+_\infty),\lambda_\infty^{\text{mix}}(\Omega^-_\infty)\}=: \tilde{\lambda}.
    $$
   Let us set $\Gamma_1^\pm=\partial \Omega^\pm_\infty \cap\Omega$ and $\Gamma_2^\pm= \partial \Omega^\pm_\infty\cap \partial\Omega$. 
Analogously, we define the sets

$$\mathcal{A}_\pm=\left\{x\in \overline{\Omega^\pm_\infty} : d(x, \Gamma_1^\pm)\ge \frac{1}{\beta}+d(x,\Gamma_2^\pm)\right\}.$$

Let us distinguish the following cases
\begin{description}
    \item[Case (1)]$\mathcal{A}_\pm\neq \emptyset$: we consider
    $$x_\pm\in \argmin_{x\in \mathcal{A}_\pm } \frac{1}{\frac{1}{\beta}+d(x,\Gamma_2^\pm)},$$
    then
    $$d(x_\pm,\Gamma^\pm_1)\ge \frac{1}{\beta}+d(x_\pm,\Gamma^\pm_2)\ge\frac{1}{\tilde\lambda}$$
    and the function $C_{\pm, \frac{1}{\tilde \lambda}}$ satisfies
    $$\abs{\nabla C_{\pm, \frac{1}{\tilde\lambda}}}=\tilde\lambda, \quad C_{\pm, \frac{1}{\tilde \lambda}}=0 \text{ on } \Gamma_1^\pm, \quad \beta\norma{C_{\pm, \frac{1}{\tilde \lambda}}}_{L^\infty(\Gamma^\pm_2)}\le\tilde \lambda.$$
    
    \item[Case (1)]$\mathcal{A}_\pm= \emptyset$: we consider
    $$x_\pm\in \argmin_{x\in N_\pm } \frac{1}{d(x,\Gamma^\pm_1)},$$
    then
    $$\frac{1}{\tilde \lambda}\le d(x_\pm,\Gamma^\pm_1)\le \frac{1}{\beta}+d(x_\pm,\Gamma^\pm_2)$$
    and the function $C_{\pm, \frac{1}{\tilde\lambda}}$ satisfies
    $$\abs{\nabla C_{\pm, \frac{1}{\tilde \lambda}}}=   \tilde \lambda, \quad C_{\pm, \frac{1}{\tilde \lambda}}=0 \text{  on } \Gamma_1^\pm, \quad \beta\norma{C_{\pm, \frac{1}{\tilde \lambda}}}_{L^\infty(\Gamma^\pm_2)}<\tilde \lambda.$$     
\end{description}

From Case 1 and 2 we deduce the existence of two points $x_+\in \Omega^+_\infty$ and $x_-\in \Omega^-_\infty$ such that
$$|x_+-x_-|\ge |x_+-y_+|+|x_--y_-|\ge d(x_+, \Gamma_1^-)+d(x_-, \Gamma_1^-)\ge \frac{2}{\tilde \lambda}$$ with $y_+\in \Gamma_1^+$ and $y_-\in \Gamma_1^-$, 
and the associated cones satisfies
$$\norma{C_{\pm, \frac{1}{\tilde\lambda}}}_{L^\infty(\Gamma^\pm_2)}<\frac{\tilde \lambda}{\beta}.$$
This means that $\displaystyle \frac{1}{\tilde \lambda} \leq s(\Omega)$ by the definition of $s(\Omega)$ and we get  
\begin{equation}
    \label{mixmaggiore}
      \liminf_p (\lambda_{2,p}(\Omega))^\frac{1}{p} \geq\max\{\lambda_\infty^{\text{mix}}(\Omega^+_\infty),\lambda_\infty^{\text{mix}}(\Omega^-_\infty)\}\geq \frac{1}{s(\Omega)},
\end{equation}
and the proof is complete.
\end{proof}

\section{Variational Characterization of \texorpdfstring{$\lambda_{2,\infty}$}{lam}}
\label{varcar_sec}

We are now in a position to prove Theorem \ref{var:carteorema}.

We recall the definition of the set $\mathcal{S}=\left\{v\in W^{1,\infty}(\Omega): \norma{v}_{L^\infty(\Omega)}=1\right\}$, and we call $u_{1,\infty}$ the limit of the first Robin $p$-eigenfunction. In \cite{AMNT} it was proved that $u_{1,\infty}$ is a positive infinity eigenfunction  and a minimum of the functional

\begin{equation*}
     \max\{\norma{\nabla u}_{L^\infty(\Omega)}, \beta \norma{u}_{L^\infty(\partial\Omega)}\},
\end{equation*}
among all ${u\in\mathcal{S}}$.
\begin{proof}[Proof of Theorem \ref{var:carteorema}]
    For all $\gamma\in \Gamma$ there exists $t_0\in [0,1]$ such that $u=\gamma(t_0)$ satisfies

    $$\norma{u_+}_{L^\infty(\Omega)}=\norma{u_-}_{L^\infty(\Omega)}=1.$$
    Let $\Omega^+=\{u>0\}$, $\Omega^-=\{u<0\}$ and let $N^+\subset\Omega^+$, $N^-\subset\Omega^-$ be two nodal domains such that
    $\norma{u_\pm}_{L^\infty(N^\pm)}=1$. Then,

    \begin{equation*}
    \begin{aligned}
        &\max\left\{\norma{\nabla u}_{L^\infty(\Omega)}, \beta \norma{u}_{L^\infty(\partial\Omega)}\right\}\ge \max\left\{\norma{\nabla u_+}_{L^\infty(\Omega)}, \beta \norma{u_+}_{L^\infty(\partial\Omega)}\right\}\ge\\
        &\max\left\{\norma{\nabla u_+}_{L^\infty(N^+)}, \beta \norma{u_+}_{L^\infty(\partial N^+)}\right\}\ge \lambda_{\infty}^{\text{mix}}(N^+),
        \end{aligned}
    \end{equation*}
    where the mixed eigenvalue $\lambda_\infty^{\text{mix}}$ is the one studied in \cite{RS} obtained by choosing $\Gamma_1^+=\partial N^+ \cap\Omega$ and $\Gamma_2^+=\partial N^+ \cap\partial\Omega$. If we replace $N^+$ with $N^-$, we obtain
    \begin{equation*}
    \max\left\{\norma{\nabla u}_{L^\infty(\Omega)}, \beta \norma{u}_{L^\infty(\partial\Omega)}\right\}\ge \lambda_\infty^{\text{mix}}(N^-).
    \end{equation*}
    
    {Arguning as in the proof of Theorem \eqref{autolimdiri}, see formula \eqref{mixmaggiore}, we have }

    $$\max\{\lambda_\infty^{\text{mix}}(N^+), \lambda_\infty^{\text{mix}}(N^-)\}\ge \lambda_{2,\infty}(\Omega).$$

    So,
    $$\lambda_{2,\infty}(\Omega)\le \max\left\{\norma{\nabla u}_{L^\infty(\Omega)}, \beta \norma{u}_{L^\infty(\partial\Omega)}\right\}\le \sup_{u\in \gamma}\max\left\{\norma{\nabla u}_{L^\infty(\Omega)}, \beta \norma{u}_{L^\infty(\partial\Omega)}\right\},$$
that gives
$$\lambda_{2,\infty}(\Omega)\le \inf_{\gamma\in \Gamma}\sup_{u\in \gamma}\max\left\{\norma{\nabla u}_{L^\infty(\Omega)}, \beta \norma{u}_{L^\infty(\partial\Omega)}\right\}.$$

To prove the other inequality, we consider a path $\gamma\in \Gamma$ such that for all $u\in \gamma$

$$\max\left\{\norma{\nabla u}_{L^\infty(\Omega)}, \beta \norma{u}_{L^\infty(\partial\Omega)}\right\}\le \lambda_{2,\infty}(\Omega).$$
By the definition of $\lambda_{2,\infty}(\Omega)$ \eqref{ess}, there exist two points $x_1,x_2\in \overline{\Omega}$ such that $$\abs{x_1-x_2}\ge \frac{2}{\lambda_{2,\infty}(\Omega)}, \quad \beta\norma{C_{i, \frac{1}{\lambda_{2,\infty}(\Omega)}}}_{L^\infty(\partial\Omega)}\le \lambda_{2,\infty}(\Omega).$$

Let $u_1$ be the limit of the first Robin $p$-eigenfunctions, and let us define, for $t\in [0,1]$, 
\begin{align*}
    \gamma_1(t)&=\max\{u_1, tC_1\},\\
    \gamma_2(t)&=\max\{(1-t)u_1, C_1\},\\
    \gamma_3(t)&= C_1-tC_2,\\
    \gamma_4(t)&=(1-t)C_1-C_2,\\
    \gamma_5(t)&=\min\{-C_2, -tu_1\},\\
    \gamma_6(t)&=\min\{-(1-t)C_2, -u_1\},
\end{align*}
the path $\gamma$ will be obtained by gluing together these six paths. 

Let us observe that, for all $t\in [0,1]$,

$$\norma{\nabla \gamma (t)}_{L^\infty(\Omega)}\le \max\{\norma{\nabla u_1}_{L^\infty(\Omega)}, \norma{\nabla C_i}_{L^\infty(\Omega)}\}\le \lambda_{2,\infty}(\Omega),$$

as 
$$\begin{gathered}
    \norma{\nabla C_i}_{L^\infty(\Omega)}=\lambda_{2,\infty}(\Omega), \\
    \norma{\nabla u_1}_{L^\infty(\Omega)}\le \max\{\norma{\nabla u_1}_{L^\infty(\Omega)}, \beta\norma{u_1}_{L^\infty(\partial\Omega)}\} {=\lambda_{1,\infty}(\Omega)}\le \lambda_{2,\infty}(\Omega),
\end{gathered}
$$
\begin{center}
    and
\end{center}
$$\beta\norma{ \gamma (t)}_{L^\infty(\partial\Omega)}\le \beta \max\{\norma{ u_1}_{L^\infty(\partial\Omega)}, \norma{C_i}_{L^\infty(\partial\Omega)}\}\le \lambda_{2,\infty}(\Omega),$$
which concludes the proof. 
\end{proof}

 As a consequence of Theorem \ref{var:carteorema}, we also have the following variational characterization
\begin{corollario}
    Let $\Omega$ be an open, bounded and Lipschitz set. Then 
\begin{equation*}
        \lambda_{2,\infty} (\Omega) = \inf_{w \in \mathcal{O}} \frac{\max\{\norma{\nabla w}_{L^\infty(\Omega)}, \beta\norma{w}_{L^\infty(\partial\Omega)}\}}{\norma{ w}_{L^\infty(\Omega)}},
\end{equation*}
where 
$$
\mathcal{O}= \{ w \in W^{1,\infty}(\Omega): \, \norma{ w_+}_{L^\infty(\Omega)}=\norma{ w_-}_{L^\infty(\Omega)}\}
$$
\end{corollario}

\section{The \texorpdfstring{$\infty$}{inf}-eigenvalue problem}
\label{infty_eigenproblem}

{
Now, with the additional assumption on the regularity of $\Omega$, we prove Theorem \ref{secondodavvero}.

\begin{proof}[Proof of Theorem \ref{secondodavvero}]
Let $u$ be the eigenfunction corresponding to $\lambda$ and let $\Omega^+=\{x \in \Omega : \, u > 0\}$ and $\Omega^-=\{x \in \Omega : \, u < 0\}$. Let us normalize the eigenfunction $u$ such that
$$
\max_{x \in \Omega^+} u(x)= \frac{1}{\lambda}.
$$
Then $u$ satisfies, in viscosity sense, 
$$
\begin{cases}
    \min\Set{\abs{\nabla u} - 1, - \Delta_\infty u}\leq 0 &\text{ in } \Omega^+\\
    \displaystyle{\frac{\partial u}{\partial \nu}} \leq 0 &\text{ on } \partial\Omega^+ \cap \partial \Omega \\
    u=0 &\text{ on } \partial\Omega^+ \cap \Omega.
\end{cases}
$$
We are in the same situation of Proposition 1 in \cite{EKNT}, hence, for all $x_0  \in \Omega \setminus\Omega^+$ we can consider  $$g_{\varepsilon, \gamma}(x) =(1+\varepsilon)\abs{ x- x_0} -\gamma \abs{x-x_0}^2 $$ as test function in the definition of viscosity subsolution, obtaining $g_{\varepsilon,\gamma}(x)\ge u$, for all $\varepsilon>0$, and $\gamma<\frac{\varepsilon}{2D(\Omega)}$, so if we pass to the limit as $\varepsilon$ goes to 0, we get 
\begin{equation}
\label{quasineum}
\abs{x-x_0} \geq u  \quad \forall x \in \overline{\Omega^+} , \quad \forall x_0 \in \overline{\Omega^-}.
\end{equation}
On the other hand, $u$ is a viscosity subsolution to

\begin{equation}
\label{comparison}
\begin{cases}
    \min\Set{\abs{\nabla u} - 1, - \Delta_\infty u}= 0 &\text{ in } \Omega^+\\
    -{\rm {min}}  \left\{| \nabla u| - \beta u, \displaystyle{-\frac{\partial u }{\partial \nu}}\right\} = 0 &\text{ on } \partial\Omega^+.
\end{cases}
\end{equation}

Hence, arguing as in Theorem 3.6 in \cite{AMNT}, we consider
for every $\varepsilon >0 $ and $0<\gamma <\frac{\varepsilon}{D(\Omega)}$, the  supersolution to \eqref{comparison} $$h_{\varepsilon, \gamma}(x) = \frac{1}{\beta} + (1+ \varepsilon) d(x , \partial \Omega^+ ) - \gamma d(x , \partial \Omega^+)^2,$$
obtaining $h_{\varepsilon,\gamma}(x)\ge u$, for all $\varepsilon, \gamma$. 
Then it follows, by letting $\varepsilon$, thus $\gamma$, to $0$, that 
\begin{equation}
    \label{quasirobin}
    \frac{1}{\beta} + d(x,\partial \Omega^+ ) \geq u(x) \qquad \forall x \in \Omega^+.
\end{equation}

Combining \eqref{quasineum} and \eqref{quasirobin}, we get
\begin{equation}
\label{quasitutto}
u(x)\leq {\rm min} \left\{d(x , \Gamma_1^+), \frac{1}{\beta} + d(x, \partial\Omega^+)\right\} \leq  {\rm min} \left\{d(x , \Gamma_1^+), \frac{1}{\beta} + d(x, \Gamma_2^+)\right\}=:g(x),
\end{equation}
where 
$$
\Gamma_1^+= \partial \Omega^+ \cap \Omega \quad \text{ and }\quad \Gamma_2^+ = \partial \Omega^+ \cap \partial\Omega.
$$
Since $g(x)$ belongs to the space $X$ defined in \eqref{def:spacex}, we have 
\begin{equation*}
    \lambda_\infty^{\text{mix}}(\Omega^+) \leq \frac{\max \left\{ \norma{\nabla g}_{L^\infty(\Omega^+) }, \beta \norma{ g}_{L^\infty(\Gamma_2^+) }\right\}}{\norma{ g}_{L^\infty(\Omega^+) }},
\end{equation*}
and so also 
\begin{equation*}
   \norma{ g}_{L^\infty(\Omega^+) } \leq \frac{\max \left\{ \norma{\nabla g}_{L^\infty(\Omega^+) }, \beta \norma{ g}_{L^\infty(\Gamma_2^+) }\right\}}{\lambda_\infty^{\text{mix}}(\Omega^+)} \leq\frac{1}{\lambda_\infty^{\text{mix}}(\Omega^+)}.
\end{equation*}
Then \eqref{quasitutto} and the latter imply
\begin{equation}
\label{parte1}
    \frac{1}{\lambda} = \max_{x \in \Omega^+}u(x) \leq  \max_{x \in \Omega^+}g(x) \le \frac{1}{\lambda_{\infty}^{\text{mix}}(\Omega^+)}.
\end{equation}

    Normalizing  $u$ such that
$$
\max_{x \in \Omega^-}( -u(x))= \frac{1}{\lambda},
$$
and arguing in the same way, we obtain

\begin{equation*}
-u(x)\leq  {\rm min} \left\{d(x , \Gamma_1^-), \frac{1}{\beta} + d(x, \Gamma_2^-)\right\}=:h(x),
\end{equation*}
where 
$$
\Gamma_1^-= \partial \Omega^- \cap \Omega \quad \text{ and }\quad \Gamma_2^- = \partial \Omega^- \cap \partial\Omega.
$$
Hence, we get 
\begin{equation}
\label{parte2}
    \frac{1}{\lambda} = \max_{x \in \Omega^-}(-u(x)) \leq  \max_{x \in \Omega^-}h(x) \le\frac{1}{\lambda_{\infty}^{\text{mix}}(\Omega^-)}.
\end{equation}
By \eqref{parte1}, \eqref{parte2} and \eqref{mixmaggiore}, we get
$$
\lambda \geq \max\{ \lambda_{\infty}^{\text{mix}}(\Omega^+),\lambda_{\infty}^{\text{mix}}(\Omega^-)\}\geq \lambda_{2,\infty}(\Omega)
,$$
as we want to show. \qedhere
\end{proof}

\begin{oss}
    In the case $\beta\le \frac{2}{D(\Omega)}$, in which the eigenvalue $\lambda_{2,\infty}(\Omega)$ coincides with the Neumann one, it is sufficient to establish 
    $$u\le d(x,\Gamma_1^+), \text{ in } \Omega^+, \quad -u\le d(x,\Gamma_1^-), \text{ in } \Omega^-,$$
    to obtain the claim. For this purpose, the convexity assumption is the only necessary one. 

    The $C^2$ assumption is needed only in the case $\beta>\frac{2}{D(\Omega)}$, to conduct the same analysis we did in \cite{AMNT}.
\end{oss}

}
\section{Some examples}
For some special shapes, we can explicitly write $s(\Omega)$ in terms of other geometric quantities of the set. 

\begin{esempio}
    Let $\Omega\subset\R^2$ be a stadium, i.e., the convex hull of two congruent disjoint balls, as represented in Figure \ref{fig:stadio}, then

\begin{equation*}
    \frac{1}{s(\Omega)}=\lambda_{2,\infty}(\Omega)=\begin{cases}
        \dfrac{1}{\frac{1}{\beta}+r(\Omega)} & \text{if } \beta\ge \frac{2}{D(\Omega)-4r(\Omega)}\\[2ex]
        \dfrac{2\beta}{1+\beta \frac{D(\Omega)}{2}} & \text{if } \frac{2}{D(\Omega)}\le\beta<\frac{2}{D(\Omega)-4r(\Omega)}\\[2ex]
        \dfrac{2}{D(\Omega)} & \text{if } \beta<\frac{2}{D(\Omega)}.
    \end{cases}
\end{equation*}

We observe that $\lambda_{2,\infty}$ is a continuous function of $\beta$ and in the case that the stadium is thin enough (i.e., $D-4r(\Omega)>0$), the second eigenvalue coincides with the first one. 
\begin{figure}[h!]
\begin{center}
\tikzset{every picture/.style={line width=0.75pt}} 

\begin{tikzpicture}[x=0.75pt,y=0.75pt,yscale=-1,xscale=1]

\draw  [color={rgb, 255:red, 155; green, 155; blue, 155 }  ,draw opacity=1 ] (360.35,179.9) .. controls (360.35,152.34) and (382.69,130) .. (410.25,130) .. controls (437.81,130) and (460.15,152.34) .. (460.15,179.9) .. controls (460.15,207.46) and (437.81,229.8) .. (410.25,229.8) .. controls (382.69,229.8) and (360.35,207.46) .. (360.35,179.9) -- cycle ;
\draw  [color={rgb, 255:red, 155; green, 155; blue, 155 }  ,draw opacity=1 ] (100.35,179.9) .. controls (100.35,152.34) and (122.69,130) .. (150.25,130) .. controls (177.81,130) and (200.15,152.34) .. (200.15,179.9) .. controls (200.15,207.46) and (177.81,229.8) .. (150.25,229.8) .. controls (122.69,229.8) and (100.35,207.46) .. (100.35,179.9) -- cycle ;
\draw  [draw opacity=0] (150.25,229.8) .. controls (150.25,229.8) and (150.25,229.8) .. (150.25,229.8) .. controls (122.5,229.8) and (100,207.46) .. (100,179.9) .. controls (100,152.34) and (122.5,130) .. (150.25,130) -- (150.25,179.9) -- cycle ; \draw   (150.25,229.8) .. controls (150.25,229.8) and (150.25,229.8) .. (150.25,229.8) .. controls (122.5,229.8) and (100,207.46) .. (100,179.9) .. controls (100,152.34) and (122.5,130) .. (150.25,130) ;  
\draw    (150.25,130) -- (410.25,130) ;
\draw  [draw opacity=0] (410.25,130) .. controls (410.25,130) and (410.25,130) .. (410.25,130) .. controls (438,130) and (460.5,152.34) .. (460.5,179.9) .. controls (460.5,207.46) and (438,229.8) .. (410.25,229.8) -- (410.25,179.9) -- cycle ; \draw   (410.25,130) .. controls (410.25,130) and (410.25,130) .. (410.25,130) .. controls (438,130) and (460.5,152.34) .. (460.5,179.9) .. controls (460.5,207.46) and (438,229.8) .. (410.25,229.8) ;  
\draw    (150.25,229.8) -- (410.25,229.8) ;
\draw [color={rgb, 255:red, 155; green, 155; blue, 155 }  ,draw opacity=1 ]   (150.25,179.9) -- (150.25,130) ;
\draw    (279.52,72.8) -- (280.5,249.8) ;
\draw [shift={(279.5,69.8)}, rotate = 89.68] [fill={rgb, 255:red, 0; green, 0; blue, 0 }  ][line width=0.08]  [draw opacity=0] (8.93,-4.29) -- (0,0) -- (8.93,4.29) -- cycle    ;
\draw    (54.35,179.9) -- (510.96,179.9) ;
\draw [shift={(513.96,179.9)}, rotate = 180] [fill={rgb, 255:red, 0; green, 0; blue, 0 }  ][line width=0.08]  [draw opacity=0] (8.93,-4.29) -- (0,0) -- (8.93,4.29) -- cycle    ;
\draw  [color={rgb, 255:red, 251; green, 0; blue, 251 }  ,draw opacity=1 ] (64.17,179.9) .. controls (64.17,132.36) and (102.71,93.83) .. (150.25,93.83) .. controls (197.79,93.83) and (236.33,132.36) .. (236.33,179.9) .. controls (236.33,227.44) and (197.79,265.98) .. (150.25,265.98) .. controls (102.71,265.98) and (64.17,227.44) .. (64.17,179.9) -- cycle ;
\draw  [color={rgb, 255:red, 255; green, 0; blue, 255 }  ,draw opacity=1 ] (324.17,179.9) .. controls (324.17,132.36) and (362.71,93.83) .. (410.25,93.83) .. controls (457.79,93.83) and (496.33,132.36) .. (496.33,179.9) .. controls (496.33,227.44) and (457.79,265.98) .. (410.25,265.98) .. controls (362.71,265.98) and (324.17,227.44) .. (324.17,179.9) -- cycle ;
\draw [color={rgb, 255:red, 255; green, 0; blue, 255 }  ,draw opacity=1 ]   (410.25,179.9) -- (472.5,238.8) ;

\draw (102.35,182.9) node [anchor=north west][inner sep=0.75pt]   [align=left] {$\displaystyle {\displaystyle \textcolor[rgb]{0.61,0.61,0.61}{r( \Omega ) =r_{2}( \Omega )}}$};
\draw (407,235) node [anchor=north west][inner sep=0.75pt]   [align=left] {$\displaystyle {\displaystyle \textcolor[rgb]{1,0,1}{s( \Omega )}}$};

\end{tikzpicture}

\end{center}
\caption{The stadium.}
\label{fig:stadio}
\end{figure}
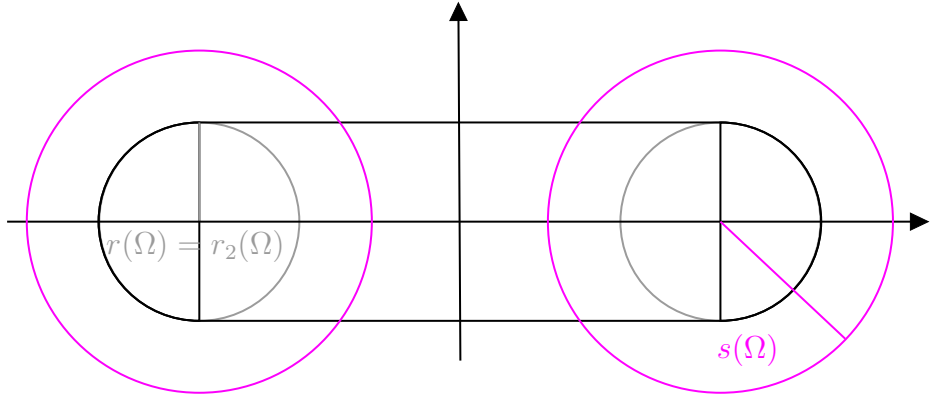
\end{esempio}

\begin{esempio}
If $\Omega\subset \R^2$ is a square, see Figure \ref{fig:square}, than 

\begin{equation*}
    \frac{1}{s(\Omega)}=\lambda_{2,\infty}(\Omega)=\begin{cases}
        \dfrac{\left(1+\frac{\sqrt{2}}{2}\right)\beta}{1+\frac{\sqrt{2}}{2}\beta \frac{D(\Omega)}{2}} & \text{if } \beta\ge \frac{2}{D(\Omega)}\\[2ex]
        \dfrac{2}{D(\Omega)} & \text{if } \beta< \frac{2}{D(\Omega)}.
    \end{cases}
\end{equation*}
    Also in this case, $\lambda_{2,\infty}(\Omega)$ is a continuous function of $\beta$. In this case, it is always true that

    $$\lambda_{1,\infty}(\Omega)<\lambda_{2,\infty}(\Omega).$$
\end{esempio}

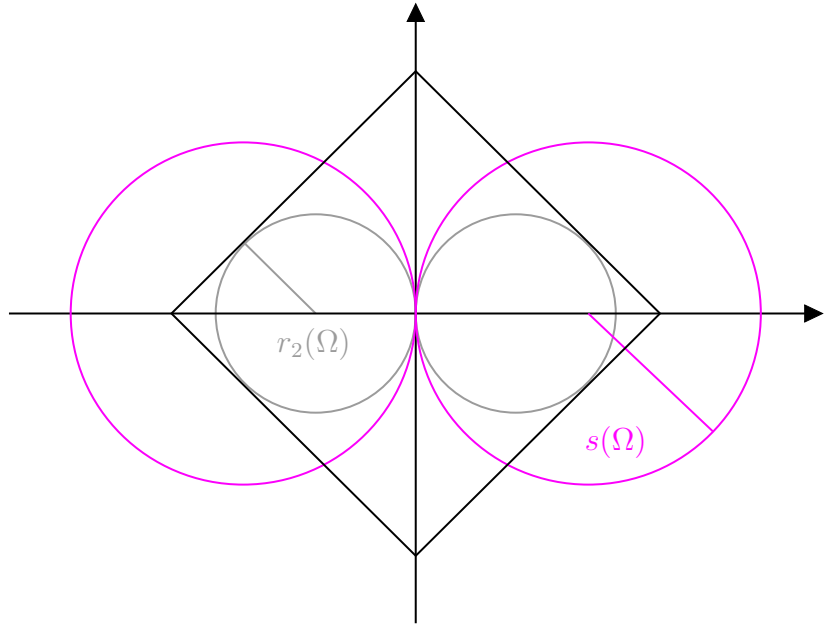
\begin{figure}[h!]
\begin{center}

\tikzset{every picture/.style={line width=0.75pt}} 

\begin{tikzpicture}[x=0.75pt,y=0.75pt,yscale=-1,xscale=1]

\draw  [color={rgb, 255:red, 155; green, 155; blue, 155 }  ,draw opacity=1 ] (311.15,164.9) .. controls (311.15,137.34) and (333.49,115) .. (361.05,115) .. controls (388.61,115) and (410.96,137.34) .. (410.96,164.9) .. controls (410.96,192.46) and (388.61,214.8) .. (361.05,214.8) .. controls (333.49,214.8) and (311.15,192.46) .. (311.15,164.9) -- cycle ;
\draw  [color={rgb, 255:red, 155; green, 155; blue, 155 }  ,draw opacity=1 ] (211.35,164.9) .. controls (211.35,137.34) and (233.69,115) .. (261.25,115) .. controls (288.81,115) and (311.15,137.34) .. (311.15,164.9) .. controls (311.15,192.46) and (288.81,214.8) .. (261.25,214.8) .. controls (233.69,214.8) and (211.35,192.46) .. (211.35,164.9) -- cycle ;
\draw [color={rgb, 255:red, 155; green, 155; blue, 155 }  ,draw opacity=1 ]   (261.25,164.9) -- (225.5,129.3) ;
\draw    (311.15,12.07) -- (311.15,320.73) ;
\draw [shift={(311.15,9.07)}, rotate = 90] [fill={rgb, 255:red, 0; green, 0; blue, 0 }  ][line width=0.08]  [draw opacity=0] (8.93,-4.29) -- (0,0) -- (8.93,4.29) -- cycle    ;
\draw    (108.2,164.9) -- (511.15,164.9) ;
\draw [shift={(514.15,164.9)}, rotate = 180] [fill={rgb, 255:red, 0; green, 0; blue, 0 }  ][line width=0.08]  [draw opacity=0] (8.93,-4.29) -- (0,0) -- (8.93,4.29) -- cycle    ;
\draw  [color={rgb, 255:red, 251; green, 0; blue, 251 }  ,draw opacity=1 ] (139,164.9) .. controls (139,117.36) and (177.54,78.83) .. (225.08,78.83) .. controls (272.61,78.83) and (311.15,117.36) .. (311.15,164.9) .. controls (311.15,212.44) and (272.61,250.98) .. (225.08,250.98) .. controls (177.54,250.98) and (139,212.44) .. (139,164.9) -- cycle ;
\draw  [color={rgb, 255:red, 251; green, 0; blue, 251 }  ,draw opacity=1 ] (311.17,164.9) .. controls (311.17,117.36) and (349.71,78.83) .. (397.25,78.83) .. controls (444.79,78.83) and (483.33,117.36) .. (483.33,164.9) .. controls (483.33,212.44) and (444.79,250.98) .. (397.25,250.98) .. controls (349.71,250.98) and (311.17,212.44) .. (311.17,164.9) -- cycle ;
\draw [color={rgb, 255:red, 255; green, 0; blue, 255 }  ,draw opacity=1 ]   (397.25,164.9) -- (459.5,224.3) ;
\draw   (311.15,43.02) -- (433.04,164.9) -- (311.15,286.79) -- (189.27,164.9) -- cycle ;

\draw (240.35,170.9) node [anchor=north west][inner sep=0.75pt]   [align=left] {$\displaystyle \textcolor[rgb]{0.61,0.61,0.61}{r_{2}( \Omega )}$};
\draw (394,220) node [anchor=north west][inner sep=0.75pt]   [align=left] {$\displaystyle {\displaystyle \textcolor[rgb]{1,0,1}{s( \Omega )}}$};

\end{tikzpicture}

\end{center}
\caption{The square.}
\label{fig:square}
\end{figure}

\subsection*{Acknowledgements}
The authors were partially supported by Gruppo Nazionale per l’Analisi Matematica, la Probabilità e le loro Applicazioni
(GNAMPA) of Istituto Nazionale di Alta Matematica (INdAM).   \\
Vincenzo Amato and Cristina Trombetti were partially supported by PRIN-PNRR 2022, P2022YFAJH: "Linear and Nonlinear PDE’S: New directions and Applications,"\\ CUP:E53D23018060001. \\
Alba Lia Masiello was partially supported by  PRIN 2022, 20229M52AS: "Partial differential equations and related geometric-functional
inequalities," CUP:E53D23005540006.

\addcontentsline{toc}{chapter}{Bibliografia}
\bibliographystyle{plain}
\bibliography{biblio}

\Addresses

\end{document}